\documentclass[11pt,reqno]{amsart}

\usepackage{amsmath}
\usepackage{amssymb}
\usepackage{tikz}
\usepackage{graphicx}
\usepackage{float}
\usepackage{color}
\usepackage{url}
\usepackage{blkarray}
\usepackage{mathdots}
\usepackage{ifpdf, hyperref}
\hypersetup{
    colorlinks,
    citecolor=black,
    filecolor=black,
    linkcolor=black,
    urlcolor=black
}

\newtheorem{thm}{Theorem}
\newtheorem{cor}[thm]{Corollary}
\newtheorem*{que*}{Question}
\newtheorem{defi}[thm]{Definition}

\newtheorem{lem}[thm]{Lemma}

\newtheorem{remark}[thm]{Remark}
\newtheorem{que}[thm]{Question}

\begin{document}
\title[Trace field degrees of Abelian differentials]{Trace field degrees of Abelian differentials}

\date{\today}
\author{Erwan Lanneau and Livio Liechti}

\address{ D\'epartement de Math\'ematiques,
Universit\'e de Fribourg,
Chemin du Mus\'ee~23,
1700 Fribourg,
Suisse
}
\email{livio.liechti@unifr.ch}

\address{
UMR CNRS 5582, 
Univ. Grenoble Alpes, CNRS, Institut Fourier, F-38000 Grenoble, France}
\email{erwan.lanneau@univ-grenoble-alpes.fr}


\begin{abstract} 
We prove that every even number~$2\le 2d \le 2g$ is realised as the degree of a Thurston--Veech pseudo-Anosov stretch factor in every connected component of every stratum of the moduli space of Abelian differentials. 
\end{abstract}
\maketitle

\section{Introduction} 

Pseudo-Anosov mapping classes first appeared in Thurston's work in connection to classification of surface homeomorphisms. Nowadays, their study is a theory by itself  combining Teichm\"uller theory, dynamics, flat geometry and number theory. A mapping class $f$ is pseudo-Anosov if and only if it asymptotically stretches every isotopy class of essential simple closed curves by a fixed factor $\lambda(f)$, with respect to any Riemannian metric. Examples arising from Anosov torus covers are abundant, and there are many other constructions, for instance using train-tracks~\cite{PP}, Rauzy induction~\cite{Rauzy,Veech}, veering triangulations~\cite{Agol}, Penner's construction~\cite{Penner} or Thurston--Veech's construction\footnote{In the literature, this construction is often called \emph{Thurston's construction}. We choose the name to include Veech since in its full generality, the construction first appeared, independently, in the two cited articles by Thurston and Veech.}~\cite{Th} and~\cite[\S 9]{Veech:construction}.

\subsection*{Algebraic degrees of stretch factors}
An important aspect of the theory of pseudo-Anosov mapping classes emerged with Fried's work and concerns the study of the stretch factor $\lambda(f)$. This is a bi-Perron algebraic integer of degree bounded above by the dimension of the Teichm\"uller space for the underlying surface. The question of realising any bi-Perron algebraic integer as a stretch factor is a major challenge in the theory. 
Despite recent advances~\cite{Pankau,Liechti:Pankau}, Fried's question remains widely open. Observe that we cannot hope for a positive answer if we fix the topology of the underlying surface: there are cubic bi-Perron number that are not realised as the stretch factor of any mapping class on a genus three surface, see the work by Thurston~\cite[Page 6]{Thurston:Last}, and more recently~\cite{Yazdi}.

\subsection*{Thurston--Veech's construction}
Given two multicurves~$\alpha = \alpha_1\cup\dots\cup\alpha_n$ and~$\beta = \beta_1 \cup \dots \cup \beta_m$ with~$n$ and~$m$ components, respectively, that fill a surface~$\Sigma$ and intersect minimally, we let~$X$ be their geometric intersection matrix. In his 1988 seminal bulletin paper~\cite{Th}, Thurston proved that all nontrivial products of multitwists in~$\langle T_\alpha, T_\beta \rangle$ except powers of conjugates of~$T_\alpha$ or~$T_\beta$ are pseudo-Anosov if the Perron-Frobenius eigenvalue of~$XX^\top$ is strictly greater than four. 
In the same article, Thurston provides the upper bound on the algebraic degree of a pseudo-Anosov stretch factor $\lambda(f)$ by the dimension of the Teichm\"uller space in general, and by $2g$ in the special case of orientable invariant foliations. He also claimed, without proof, that {``\em the examples of~\cite[Theorem~7]{Th} show that this bound is sharp"}. The referenced examples are exactly the examples nowadays known as Thurston's construction, described above. More recently, Strenner~\cite{Strenner:algebraic} answered the question of which degrees appear for a pseudo-Anosov on a genus $g$ surface, including all nonorientable surfaces, by using Penner’s construction~\cite{Penner}.

\subsection*{Main results}
In this paper, we justify Thurston's remark for pseudo-Anosov mapping classes with orientable invariant foliations. In fact, we obtain a stronger result justifying Thurston's remark in every connected component of every strata of Abelian differentials, not just for a given genus.

\begin{thm}
\label{thm:stretch}
Every even integer~$2\le 2d \le 2g$ is realised as the degree of a stretch factor of a product of two affine multitwists on a surface in every connected component of every stratum of Abelian differentials on Riemann surfaces of genus~$g$.
\end{thm}

The terminology ``connected component'' can be skipped on a first reading of this paper, and we refer to~\S \ref{subsec:cc} and~\cite{KZ} for more details. A stratum is the set of Abelian differentials having prescribed singularity multiplicities $(k_1,\dots,k_n)$, where $\sum_{i=1}^n k_i=2g-2$.

The extension field $K=\mathbb Q(\lambda+\lambda^{-1})$ is important in Teichm\"uller dynamics and is called the trace field. It is an invariant of Abelian differentials and has degree at most $g$ over $\mathbb Q$~\cite{Gutkin:Judge,Kenyon:Smillie}. We will deduce Theorem~\ref{thm:stretch} from the following result asserting that choosing a connected component of a stratum of Abelian differentials poses no restriction on the degree of trace fields.
\begin{thm}
\label{thm:trace}
Every integer~$1\le d\le g$ is realised as the degree of the trace field of a product of two affine multitwists on a surface in every connected component of every stratum of Abelian differentials on Riemann surfaces of genus~$g$.
\end{thm}
Theorem~\ref{thm:trace} completely answers the question about trace field degrees, per components of strata of Abelian differentials. However Theorem~\ref{thm:stretch} is not quite complete since mapping classes with odd degree stretch factors can arise as product of two affine multitwists. We leave open the following question:
\begin{que}
For a given connected component $\mathcal C$ of a stratum of Abelian differentials on Riemann surfaces of genus~$g$, which odd integer $3\le d\le g$ can arise as the degree of a stretch factor of a product of two affine multitwists on a surface belonging to $\mathcal C$?
\end{que}
Stretch factor degrees and trace field degrees are closely related. Since~$\lambda$ is a root of the polynomial $t^2-(\lambda+\lambda^{-1})t+1$, the degree of~$\lambda$ over~$K$ is either one or two. The degree one case corresponds to pseudo-Anosov homeomorphisms with vanishing SAF invariant by a result of Calta and Schmidt~\cite{CS13}, see also Strenner's article~\cite{Strenner:saf} and~\cite{AY} for the first known example. As a key step in proving Theorem~\ref{thm:stretch} we present a novel nonsplitting criterion stating that the degree of the field extension~$\mathbb{Q}(\lambda) : \mathbb{Q}(\lambda+\lambda^{-1})$ equals two under certain conditions in  Thurston--Veech's construction, see Theorem~\ref{thm:criterion} in~\S \ref{sec:nonsplitting}.

Another variation of Thurston's remark concerns subgroups of the mapping class group. The following question was asked to us by Dan~Margalit. Instead of fixing a stratum,
one may fix a subgroup: {\em which algebraic degrees are attained in various infinite index subgroups of the mapping class group, such as the Torelli group or other normal subgroups?} For an integer $p$, the level $p$ congruence subgroup $\mathrm{Mod}_g(p)$ is the subgroup of the mapping class group consisting of mapping classes that act trivially on $H_1(\Sigma; \mathbb Z/p)$. Since an affine multitwist is acting as a product of commuting transvection matrices on the homology, by taking a suitable power, we deduce: 
\begin{cor}\label{cor:congruence}
For every $p>1$, every integer~$1\le d\le g$ is realised as the degree of the trace field of a pseudo-Anosov mapping class in the level $p$ congruence subgroup $\mathrm{Mod}_g(p)$.
\end{cor}
\begin{proof}[Proof of Corollary~\ref{cor:congruence}]
Fix an integer~$1\le d\le g$ and let $\alpha = \alpha_1\cup\dots\cup\alpha_n$ and $\beta = \beta_1 \cup \dots \cup \beta_m$ be two multicurves, with $n$ and $m$ components respectively, given by Theorem~\ref{thm:trace}. On the homology level, the action of the Dehn twist $T_{\alpha_i}$ along the curve $\alpha_i$ is given by $[\gamma] \mapsto [\gamma] + i(\gamma,\alpha_i)[\alpha_i]$, where $i(\cdot,\cdot)$ is the algebraic intersection form on $H_1(\Sigma; \mathbb Z)$. In particular one sees that $T^p_{\alpha_i} \in \mathrm{Mod}_g(p)$. Since $\alpha_i$ are pairwise disjoint, $T_{\alpha_i}$ are pairwise commuting mapping classes, and $T^p_\alpha=(T_{\alpha_1}\cdots T_{\alpha_n})^p=T^p_{\alpha_1}\cdots T^p_{\alpha_n} \in \mathrm{Mod}_g(p)$. Therefore the group $\langle T^p_\alpha, T^p_\beta \rangle$ is fully contained in $\mathrm{Mod}_g(p)$. Since all pseudo-Anosov mapping classes in $\langle T_\alpha, T_\beta \rangle$ have the same trace field, this leads to the result.
\end{proof}
\subsection*{Square-tiled surfaces}
The case when $[\mathbb{Q}(\lambda+\lambda^{-1}):\mathbb Q]=1$ plays a special role in Teichm\"uller theory, and our theorems are  well-known in this context. The translation surfaces admitting such pseudo-Anosov maps are also called arithmetic surfaces, or square-tiled surfaces, since they are torus coverings~\cite{Gutkin:Judge}. In particular, this implies that the field extension $\mathbb{Q}(\lambda) : \mathbb{Q}(\lambda+\lambda^{-1})$ has degree two.

\subsection*{Outline of the proof of the main results}
Let~$\mathcal{H}(k_1, k_2, \dots, k_m)$ be a given stratum of Abelian differentials in genus~$g$. Fix some $2 \leq d\leq g$. This is the degree of a trace field we want to construct. In Thurston--Veech's construction, the stretch factor~$\lambda$ of~$T_\alpha \circ T_\beta$ is related to the geometric intersection matrix of~$\alpha$ and~$\beta$ as follows: the number $\lambda + \lambda^{-1} + 2$ equals the Perron-Frobenius eigenvalue of~$XX^\top$. In order to control the degree of~$\lambda+\lambda^{-1}$, we therefore need to control the degree of the Perron-Frobenius eigenvalue of~$XX^\top$. Roughly, our strategy consists of the following four steps.
\smallskip

\emph{Step 1) construct examples.} For positive integers $y$, $y_i$, $i=1,\dots,g-1$, we start by constructing a square-tiled surface $(X,\omega) \in \mathcal{H}(k_1, k_2, \dots, k_m)$ depending on~$y$, $y_i$. We think of the numbers $y$, $y_i$ as variables that we specify in the following. Applying Thurston--Veech's construction using the core curves of the horizontal and vertical annuli of~$(X,\omega)$ gives us a $g\times g$ matrix $XX^\top$. 
\smallskip

\emph{Step 2) specify the~$y_i$.} 
The characteristic polynomial $p_g(t,y)\in\mathbb{Z}[t,y]$ of the matrix $XX^\top$ satisfies $p_g(t,y)=(t-1)^{g-d}p_d(t,y)$ if we set~$g-d+1$ of the~$g-1$ parameters~$y_i$ equal to~$1$. 
Furthermore, if all the other $y_i$ are pairwise different, then $p_d(t,y)$ is shown to be irreducible in $\mathbb{Z}[t,y]$ in~\S\ref{sec:calculating}.
\smallskip

\emph{Step 3) specify~$y$.}
Hilbert's irreducibility theorem~\cite{Lang} furnishes infinitely many integer specifications of $y$ such that $p_d(t,y)\in\mathbb{Z}[t]$ is irreducible. By our construction, all these choices of parameters correspond to surfaces in $\mathcal{H}(k_1, k_2, \dots, k_m)$. Furthermore,  the trace field is generated by the Perron-Frobenius eigenvalue of~$XX^\top$, which has degree $d$ as desired.
\smallskip

\emph{Step 4) Apply the nonsplitting criterion.} Finally, we apply Theorem~\ref{thm:criterion} to deduce that the stretch factor~$\lambda$ of~$T_\alpha \circ T_\beta$ is of degree~$2d$ for all the specifications of~$y_i$ and~$y$ as above. 
\medskip

This description of the strategy does not yet take into account the connected components we want to reach, but basically the same idea can be applied in order to deal with all connected components. However, we need to take variations of the families of examples we consider in order to find surfaces belonging to all of them.  This is dealt with in~\S\ref{sec:cc}.

\subsection*{Acknowledgments} This collaboration began during a visit of the second author in Grenoble. This work has been partially supported by the LabEx PERSYVAL-Lab (ANR-11-LABX-0025-01) funded by the French program Investissement d’avenir. The authors thank Jean-Claude Picaud for discussions on Thurston's construction and stretch factor degrees, and Dan Margalit and Curt McMullen for feedback on an earlier version of this article. The authors also thank the anonymous referees for their helpful comments and suggestions.

\section{A nonsplitting criterion} 
\label{sec:nonsplitting}

The goal of this section is to present an algebraic criterion that allows us to deduce that the degree of the field 
extension~$\mathbb{Q}(\lambda) : \mathbb{Q}(\lambda+\lambda^{-1})$ equals two for certain products of multitwists. 
Let~$\alpha = \alpha_1 \cup \dots \cup \alpha_n$ and~$\beta = \beta_1 \cup \dots \cup \beta_m$ be two 
multicurves with~$n$ and~$m$ components, respectively, that fill a surface~$\Sigma$ and intersect minimally. 
Let~$X$ be their geometric intersection matrix, that is, the~$n\times m$ matrix whose~$ij$-th coefficient equals the 
geometric intersection number of~$\alpha_i$ and~$\beta_j$. We assume that the Perron-Frobenius eigenvalue~$\mu^2$ 
of~$XX^\top$ is of degree~$d$. Furthermore, we let~$\Omega = \left(\begin{smallmatrix}0 & X \\ X^\top & 0\end{smallmatrix}\right)$. 
For a symmetric matrix $A$, we denote by $\sigma(A)$ its signature {\em i.e.}  the number of positive eigenvalues minus the number of negative eigenvalues. We will also denote by  $\mathrm{null}(A)$ its nullity {\em i.e.} the dimension of its kernel.

\begin{lem} 
\label{dissprops}
The following properties hold.
\begin{enumerate} 
\item 
\label{itemone}
The number~$\sigma(\Omega + 2I) + \mathrm{null}(\Omega + 2I)$ equals the number of eigenvalues of~$\Omega$ in the interval~$[-2,2]$.
\item 
\label{itemtwo}
The eigenvalues~$\lambda_i$ of~$M = \begin{pmatrix}I & X \\ 0 & I\end{pmatrix}\begin{pmatrix}I & 0 \\ -X^\top & I\end{pmatrix}$ 
are related to the eigenvalues~$\mu_i$ of~$\Omega$ by the equation~$\mu_i^2 = 2 - \lambda_i - \lambda_i^{-1}$. 
\end{enumerate}
\end{lem}

\begin{proof}
The first property is exactly Lemma~3.7 in~\cite{Ldiss}. The second property is Proposition~3.3(b) in~\cite{Ldiss}; 
as the proof in~\cite{Ldiss} does not explicitly deal with the case where~$M$ is not diagonalisable, we present a complete argument here. 
We first calculate
\[
M =  \begin{pmatrix}I & X \\ 0 & I\end{pmatrix}\begin{pmatrix}I & 0 \\ -X^\top & I\end{pmatrix} = \begin{pmatrix} I - XX^\top & X \\ -X^\top & I \end{pmatrix}
\]
and note that its inverse is given by
\[
M^{-1} = \begin{pmatrix} I & -X \\ X^\top & I - X^\top X \end{pmatrix}.
\]
One directly verifies the equation~$\Omega^2 = 2I - M - M^{-1}$. In order to obtain the same equation for all the eigenvalues 
(counting multiplicity), we change basis such that~$M$ is in Jordan normal form. Note that in the new basis, also the matrix~$M^{-1}$ 
becomes a block diagonal matrix, where all the blocks are of upper triangular form and correspond to the Jordan blocks of~$M$. 
In particular, also the matrix~$\Omega^2$ becomes upper triangular in the new basis, and the equation for the 
eigenvalues,~$\mu_i^2 = 2 - \lambda_i + \lambda_i^{-1}$ (counting multiplicity), can be read off from the diagonal entries of the 
matrix equation. 
\end{proof}

Our criterion for the construction of pseudo-Anosov maps with stretch factors of controlled degree is the following.

\begin{thm}
\label{thm:criterion}
Let~$\alpha$ and~$\beta$ be two 
multicurves with~$n$ and~$m$ components, respectively, that fill a surface~$\Sigma$ and intersect minimally. 
Let~$X$ be their geometric intersection matrix and assume that the Perron-Frobenius eigenvalue~$\mu^2$ 
of~$XX^\top$ is of degree~$d$. Furthermore, set~$\Omega = \left(\begin{smallmatrix}0 & X \\ X^\top & 0\end{smallmatrix}\right)$. 
If we have~$n+m>\sigma(\Omega + 2I) + \mathrm{null}(\Omega + 2I) > m+n-2d$, then 
the mapping class~$T_\alpha \circ T_\beta$ is pseudo-Anosov with stretch factor~$\lambda$ of degree~$2d$.
\end{thm}

\begin{remark} \emph{This criterion is particularly strong in case~$n=m=d$, that is, when~$\alpha$ and~$\beta$ have the same 
number of components and if the characteristic polynomial of the matrix~$XX^\top$ is irreducible. 
In this case,~$$2n >\sigma(\Omega + 2I) + \mathrm{null}(\Omega + 2I) > 0$$ is sufficient to ensure that 
the mapping class~$T_\alpha \circ T_\beta$ is pseudo-Anosov with stretch factor~$\lambda$ of degree~$2d$.}
\end{remark}

\begin{proof}
We first ensure that the mapping class~$T_\alpha \circ T_\beta$ is pseudo-Anosov. 
If~$n+m>\sigma(\Omega + 2I) + \mathrm{null}(\Omega + 2I)$, then~$\Omega$ has an eigenvalue outside the interval~$[-2,2]$ 
by~(\ref{itemone}) of Lemma~\ref{dissprops}. In particular, the dominating eigenvalue~$\mu$ of~$\Omega$ is larger than~$2$ and the matrix product
~$\left(\begin{smallmatrix} 1 & \mu \\ 0 & 1\end{smallmatrix}\right)\left(\begin{smallmatrix} 1 & 0 \\ -\mu & 1\end{smallmatrix}\right)$ is 
hyperbolic, as its trace~$2-\mu^2$ is larger than~$2$ in modulus. Hence, the mapping class~$T_\alpha \circ T_\beta$ is 
pseudo-Anosov by Thurston--Veech's construction~\cite{Th,Veech:construction}. 

Now, let~$\lambda$ be the stretch factor of the mapping class~$T_\alpha \circ T_\beta$. By Thurston--Veech's construction, 
we have~$\lambda + \lambda^{-1} = \mu^2 -2 $. In particular, we directly observe~$\mathbb{Q}(\lambda+\lambda^{-1}) = \mathbb{Q}(\mu^2).$
Furthermore, the degree of the field extension~$\mathbb{Q}(\lambda):\mathbb{Q}(\lambda+\lambda^{-1})$ is either~$1$ or~$2$. 
It equals~$2$, which is what we want to show, exactly if~$\lambda$ and~$\lambda^{-1}$ are Galois conjugates. 

We now finish the proof by arguing that~$\lambda$ and~$\lambda^{-1}$ are indeed Galois conjugates. By (\ref{itemtwo}) of Lemma~\ref{dissprops}, the dilatation~$\lambda$ is also the leading eigenvalue of~$-M$, where M is the matrix product given in the (\ref{itemtwo}) of 
Lemma~\ref{dissprops}. In particular, the Galois conjugates of~$\lambda$ are among the eigenvalues~$-\lambda_i$ of the matrix~$-M$.
These eigenvalues are in turn related to the eigenvalues~$\mu_i$ of~$\Omega$ by the equation~$\mu_i^2 = 2+\lambda_i + \lambda_i^{-1}$, 
again by Lemma~\ref{dissprops}. Since we have~$\sigma(\Omega + 2I) + \mathrm{null}(\Omega + 2I)>n+m-2d$, 
the matrix~$\Omega$ has at most~$2d-1$ eigenvalues outside the interval~$[-2,2]$. Via the correspondence in Lemma~\ref{dissprops}, 
the matrix~$-M$ hat at most~$2d-1$ eigenvalues that do not lie on the unit circle. In particular, one of the~$2d$ Galois conjugates 
of~$\lambda$ or~$\lambda^{-1}$ (including~$\lambda$ and~$\lambda^{-1}$ themselves) must be on the unit circle by the pigeonhole principle. 
Thus the minimal polynomial of~$\lambda$ or~$\lambda^{-1}$ (and hence of both) is reciprocal and it follows 
that~$\lambda$ and~$\lambda^{-1}$ are Galois conjugates. 
\end{proof}

\section{Strata of Abelian differentials} 
\label{sec:stratum}

In this section, we present a proof of our main results, Theorem~\ref{thm:trace} and Theorem~\ref{thm:stretch}, for each stratum. 
We postpone the more intricate analysis of the connected components to \S\ref{sec:cc}. 
\medskip

Let~$\mathcal{H}(k_1, k_2, \dots, k_m)$ be a stratum of Abelian differentials. Recall that the number of odd~$k_i$ must itself be even, say~$2l$.
Furthermore, if~$g$ is the genus of the underlying topological surface, we have the equality~$2g-2=\sum_{i=1}^m k_i$.

\subsection{Constructing a surface}
\label{sec:construct}
We start by constructing a square-tiled surface.
First, we ensure that we land in the stratum~$\mathcal{H}(k_1, k_2, \dots, k_m)$. 
We start out with a long horizontal square-tiled surface with some large number~$y^2-g+1$ of squares and opposite side identifications, 
see Figure~\ref{horizontal}. 
\begin{figure}[h]
\begin{center}
\def\svgwidth{190pt}
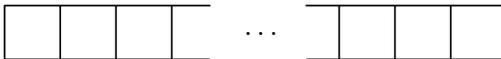
\caption{A horizontal square-tiled surface.}
\label{horizontal}
\end{center}
\end{figure}
The surface obtained by identifying the sides is a torus, and there are no singularities of the flat structure.
We can add an angle of~$4\pi$ to some marked point by inserting a vertical strip of~$y_i + 1$ square tiles, 
as in Figure~\ref{addvertical}. 
\begin{figure}[h]
\begin{center}
\def\svgwidth{240pt}
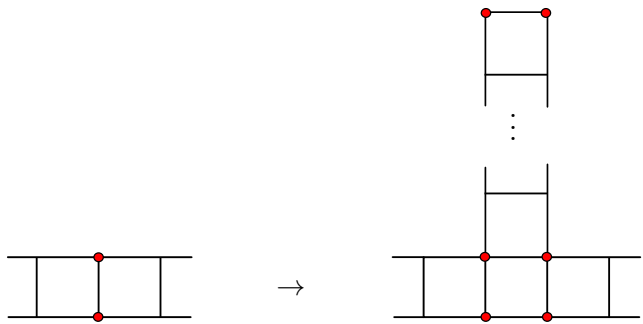
\caption{Inserting a vertical strip of squares creates a cone point with angle~$6\pi$ out of a marked point.}
\label{addvertical}
\end{center}
\end{figure}
We treat the~$y_i \ge 1$ as variables that we will need to specify later on.
This operation can be repeated in order to add an integer multiple of~$4\pi$ to the angle around any cone point or marked point. 
For example, Figure~\ref{add2vertical} indicates how to insert another vertical strip of square tiles in order 
to add~$4\pi$ to the cone angle around a cone point with angle~$6\pi$. 
\begin{figure}[h]
\begin{center}
\def\svgwidth{280pt}
\begingroup%
  \makeatletter%
  \providecommand\color[2][]{%
    \errmessage{(Inkscape) Color is used for the text in Inkscape, but the package 'color.sty' is not loaded}%
    \renewcommand\color[2][]{}%
  }%
  \providecommand\transparent[1]{%
    \errmessage{(Inkscape) Transparency is used (non-zero) for the text in Inkscape, but the package 'transparent.sty' is not loaded}%
    \renewcommand\transparent[1]{}%
  }%
  \providecommand\rotatebox[2]{#2}%
  \newcommand*\fsize{\dimexpr\f@size pt\relax}%
  \newcommand*\lineheight[1]{\fontsize{\fsize}{#1\fsize}\selectfont}%
  \ifx\svgwidth\undefined%
    \setlength{\unitlength}{489.75011474bp}%
    \ifx\svgscale\undefined%
      \relax%
    \else%
      \setlength{\unitlength}{\unitlength * \real{\svgscale}}%
    \fi%
  \else%
    \setlength{\unitlength}{\svgwidth}%
  \fi%
  \global\let\svgwidth\undefined%
  \global\let\svgscale\undefined%
  \makeatother%
  \begin{picture}(1,0.73834537)%
    \lineheight{1}%
    \setlength\tabcolsep{0pt}%
    \put(0,0){\includegraphics[width=\unitlength]{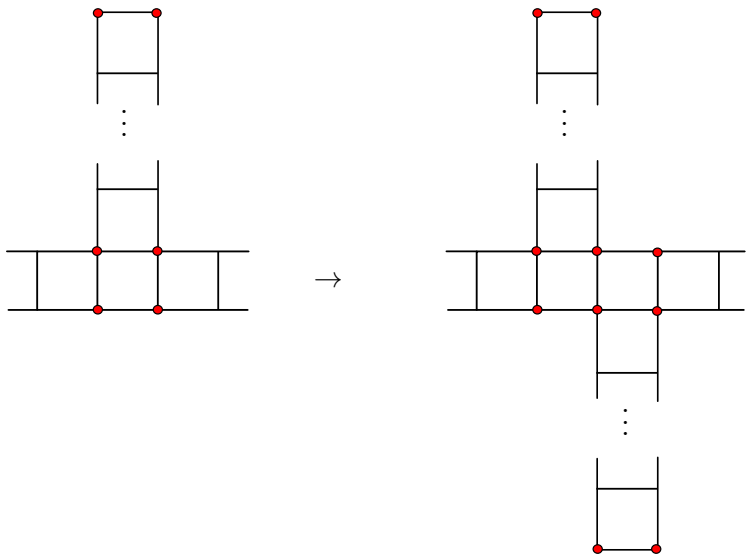}}%
    \put(0.74895173,0.56565494){\color[rgb]{0,0,0}\makebox(0,0)[lt]{\lineheight{1.25}\smash{\begin{tabular}[t]{l}$\vdots$\end{tabular}}}}%
    \put(0.41618891,0.36150213){\color[rgb]{0,0,0}\makebox(0,0)[lt]{\lineheight{1.25}\smash{\begin{tabular}[t]{l}$\to$\end{tabular}}}}%
    \put(0.15477117,0.56565494){\color[rgb]{0,0,0}\makebox(0,0)[lt]{\lineheight{1.25}\smash{\begin{tabular}[t]{l}$\vdots$\end{tabular}}}}%
    \put(0.83148665,0.16123379){\color[rgb]{0,0,0}\makebox(0,0)[lt]{\lineheight{1.25}\smash{\begin{tabular}[t]{l}$\vdots$\end{tabular}}}}%
  \end{picture}%
\endgroup%

\caption{A vertical strip of squares can be inserted in order to add another~$4\pi$ to the angle around a cone point.}
\label{add2vertical}
\end{center}
\end{figure}
Iterating this procedure, we can reach all strata with even multiplicities. 

In order to create odd multiplicities, we insert an L-shaped square-tiled surface 
with~$y_i+1$ tiles, as shown in Figure~\ref{addodd}.
\begin{figure}[h]
\begin{center}
\def\svgwidth{260pt}
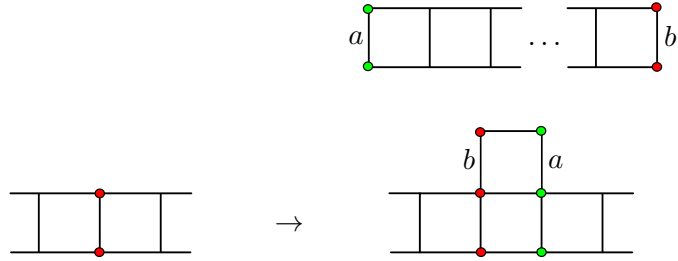
\caption{Inserting an L-shaped square-tiled surface creates two cone angles of~$4\pi$ out of one marked point.}
\label{addodd}
\end{center}
\end{figure}
This creates two cone points of angle~$4\pi$, which is multiplicity one. Recall that there must be an even number~$2l$ of odd 
multiplicities~$k_i$, so we can repeat this step~$l$ times to have the right number of odd multiplicities, and then successively 
add two to the multiplicities by inserting vertical strips as above, until we reach the stratum~$\mathcal{H}(k_1, k_2, \dots, k_m)$. 
Following this procedure, we need to add a total of~$l$ L-shapes and ~$g-l-1$ vertical strips.

\subsection{Calculating the polynomial}
\label{sec:calculating}
The square-tiled surface we construct in Section~\ref{sec:construct} naturally decomposes into horizontal and vertical annuli 
that are one square wide. Let~$X$ be the intersection matrix for the core curves~$\alpha_i$ of the horizontal annuli and the core 
curves~$\beta_j$ of the vertical annuli. We index the rows by horizontal curves and the columns by vertical curves. 
We now describe the matrix~$XX^\top$. Since the curves~$\alpha_i$ and~$\beta_j$ pairwise intersect in a tree-like fashion, 
we use the following way of looking at the computation. The~$i$-th diagonal coefficient equals the number of vertical curves 
intersecting the~$i$-th horizontal curve~$\alpha_i$. Furthermore, an off-diagonal~$ij$-th coefficient is equal to~$1$ if there 
exists a vertical curve intersecting both horizontal curves~$\alpha_i$ and~$\alpha_j$. Otherwise, it equals~$0$. 

In order to write down the matrix~$XX^\top$, we quickly recall our construction. We have one horizontal curve that we start with. 
It intersects~$y^2$ vertical curves. We further have one horizontal curve for each L-shaped surface we inserted, of which there are~$l$ in total. 
These curves respectively intersect~$y_i$ vertical curves, for~$i=1,\dots,l$, and are linked to the starting horizontal curve via an 
intersecting vertical curve. 

For example, if we insert two L-shaped surfaces with~$y_1+1$ and~$y_2+1$ tiles, respectively, 
we obtain the matrix
$$ XX^\top = 
\left(
\begin{array}{ccc}
y^2 & 1 & 1 \\
1 & y_1 & 0 \\
1 & 0 & y_2
\end{array}
\right)$$
with characteristic polynomial obtained by developing the first column of the matrix~$tI - XX^\top$:
\begin{align*}
(t-y^2)(t-y_1)(t-y_2) &- (t-y_2) -(t-y_1) =  \\
&= -y^2 \prod_{i=1}^2 (t-y_i) + t\prod_{i=1}^2 (t-y_i) - \sum_{i=1}^2 \prod_{j\ne i} (t-y_i).
\end{align*}
It is straightforward to generalise the last form of the characteristic polynomial to an arbitrary number~$l$
of inserted L-shaped surfaces. 

Conveniently, the form of the characteristic polynomial turns out to be basically the same even if we insert vertical strips, 
but this needs a more careful calculation. 
We first describe the coefficients of the matrix~$XX^\top$ we get from inserting vertical strips: 
for each vertical surface we insert, we get another~$y_i$ horizontal curves, 
all intersecting a single vertical curve that also intersects the starting horizontal curve.
Here,~$i$ runs from~$l+1$ to~$g-1$. We present the matrix using parameters $b, b_i \in \mathbb R$. 
These parameters are helpful in the proof of Lemma~\ref{irred}, and later in \S\ref{sec:non:hyp:spin0}. 
For the purpose of the calculation of~$XX^\top$ in this section, we simply have~$b = b_i=1$ for all~$i$.  
We write~$\mathbf{b}_{n\times m}$ for the~$n\times m$ matrix with all entries equal to~$b\in \mathbb R$. 
In case~$n=m$, we simplify and write~$\mathbf{b}_n$.

\begin{defi}
\label{def:XX}
For parameters~$b,b_i \in \mathbb{R}$,~$i=1,\dots, l$, we consider the matrix
$$
XX^\top=
\left(
\begin{array}{cccc|cccc}
y^2 & b_1 & \cdots & b_{l} & \mathbf{b}_{1\times y_{l+1}} & \mathbf{1}_{1\times y_{l+2}} & \cdots & \mathbf{1}_{1\times y_{g-1}} \\
b_1 & y_1 &  &  & & & \\
\vdots &  & \ddots &  & & & \\
b_l &  &  & y_l & & & \\
\hline
\mathbf{b}_{y_{l+1}\times 1} & & & & \mathbf{1}_{y_{l+1}} & &\\
\mathbf{1}_{y_{l+2}\times 1} & & & & & \mathbf{1}_{y_{l+2}} & &\\
\vdots & & & & && \ddots &\\
\mathbf{1}_{y_{g-1}\times 1} & & & & & && \mathbf{1}_{y_{g-1}}
\end{array}
\right).
$$
\end{defi}

For the characteristic polynomial of the matrix~$XX^T$, we have the following result.

\begin{lem}
\label{charpoly}
The characteristic polynomial of~$XX^\top$ equals
$$p(t,y,\mathbf{y})=
t^a\left(-y^2\prod_{i=1}^{g-1}(t-y_i) + t\prod_{i=1}^{g-1}(t-y_i) - \sum_{i=1}^{g-1} c_i\prod_{j\ne i}(t-y_j)   \right),
$$
where~$a=\sum_{i=l+1}^{g-1}(y_i-1)$, $c_{l+1}= y_{l+1}b^2$, $c_i=y_i$ for~$i\ge l+2$ and~$c_i=b_i^2$ otherwise.
\end{lem}

\begin{proof}
This calculation is slightly tedious, but obtained in a fairly straightforward manner by developing the first 
column of~$(tI - XX^\top)$. We begin by observing that the determinants of the $y_i\times y_i$ matrices
\[
tI_{y_i} -\mathbf{1}_{y_i} = \left(
\begin{array}{cccc}
t-1 & -1 & \cdots & -1  \\
-1 & t-1 &  &  \\
\vdots &  & \ddots &  \\
-1 &  &  & t-1
\end{array}
\right),
\]

\[
M_{y_i}(t) = \left(
\begin{array}{cccc}
-1 & -1 & \cdots & -1  \\
-1 & t-1 &  &  \\
\vdots &  & \ddots &  \\
-1 &  &  & t-1
\end{array}
\right)
\]
are respectively given by the polynomials $t^{y_i-1}(t-y_i)$ and $-t^{y_i-1}$. 
The former calculation follows by inspecting the eigenvalues of the matrix~$\mathbf{1}_{y_i}$, 
and the latter is derived by solving the equation~$$\det(tI_{y_i} -\mathbf{1}_{y_i}) = t\det(tI_{y_i-1} -\mathbf{1}_{y_i-1}) + \det(M_{y_i}(t)).$$
We note that changing the diagonal coefficient~$(-1)$ of the matrix~$M_{y_i}(t)$ with some other diagonal coefficient~$(t-1)$ does not change the determinant. This will be used later on in the calculation.
\smallskip

Now, by developing the first column of~$(tI - XX^\top)$, we get that the characteristic polynomial of~$XX^\top$ has the following summands. 
The first summand (obtained by deleting the first row and the first column when developing) equals
\[
(t-y^2) \prod_{i=1}^l (t-y_i) \prod_{i= l + 1}^{g-1} \det(tI_{y_i} -\mathbf{1}_{y_i}) = t^a(t-y^2)\prod_{i=1}^{g-1}(t-y_i),
\]
where~$a= \sum_{i=l+1}^{g-1}(y_i -1)$. 
The rest of the summands are obtained as follows. Assume that in the development we delete the first column and the~$k$-th row, 
where~$k\ge 2$. We have to take the determinant of the matrix obtained by deleting the~$k$-th row 
of the matrix
$$
\left(
\begin{array}{ccc|cccc}
-b_1 & \cdots & -b_l   & -\mathbf{b}_{1\times y_{l+1}} & -\mathbf{1}_{1\times y_{l+2}} & \cdots & -\mathbf{1}_{1\times y_{g-1}} \\
 t-y_1&    &  & & & \\
 & \ddots   &  & & & \\
 &  & t- y_l   & & & \\
\hline
 & & & tI_{y_{l+1}}-\mathbf{1}_{y_{l+1}} & &\\
  & & & & tI_{y_{l+2}}- \mathbf{1}_{y_{l+2}} & &\\
  & & & && \ddots &\\
  & & & & && tI_{y_{g-1}}-\mathbf{1}_{y_{g-1}}
\end{array}
\right).
$$
After switching adjacent rows (a total of~$k-2$ times) to move the first row to be the $k-1$st one, 
the matrix obtained is almost of block diagonal form and we can read off the determinant. 
For the rows~$k=2, \dots, l+1$, we obtain the summand
\begin{gather*}
 (-b_{k-1})(-1)^{1+k}(-1)^{k-2} \left( \prod_{j \ne {k-1},~1\le j \le l} (t-y_j) \prod_{i= l + 1}^{g-1} \det(tI_{y_i} -\mathbf{1}_{y_i}) \right) (-b_{k-1}) = \\
=  - b_{k-1}^2 t^a \prod_{j\ne k-1} (t-y_j).
\end{gather*}
For the rows~$k> l+y_{l+1}$, we obtain summands of the form
\begin{gather*}
 (-1)(-1)^{1+k}(-1)^{k-2}  \left(\prod_{j=1}^l (t-y_j) \prod_{j \ne i,~l+1 \le j \le g-1}\det(tI_{y_j} -\mathbf{1}_{y_j}) \right) \det (M_{y_i}(t)) = \\
=  - t^a \prod_{j\ne i}(t-y_j).
\end{gather*}
Here, we assume for the calculation that the~$k$-th row intersects the diagonal block~$tI_{y_{i}}- \mathbf{1}_{y_{i}} $, where~$i\ge l+2$.
There are a total of~$y_i$ summands of this type. If the~$k$-th row intersects the block~$tI_{y_{l+1}}- \mathbf{1}_{y_{l+1}} $, 
the corresponding constant vectors of the first row and the first column have coefficients~$b\in\mathbb{R}$. 
In this case, we obtain~$y_{l+1}$ times the summand
$$ - b^2 t^a \prod_{j\ne l+1} (t-y_j).$$
Adding all summands, we finally obtain the polynomial
$$
t^a\left( (t-y^2)\prod_{i=1}^{g-1}(t-y_i) - \sum_{i=1}^{g-1} c_i\prod_{j\ne i}(t-y_j)   \right),
$$
where~$a=\sum_{i=l+1}^{g-1}(y_i-1)$, $c_{l+1}= y_{l+1}b^2$, $c_i=y_i$ for~$i\ge l+2$ and~$c_i=b_i^2$ otherwise. 
\end{proof}

\begin{lem}
\label{irred}
Let~$k\ge 1$, and let~$y_i, c_i \in \mathbb{Z}$ for~$i=1, \dots, k$ such that all~$y_i$ are pairwise distinct and all~$c_i$ are positive.
Then the polynomial
$$p(t,y)=
-y^2\prod_{i=1}^{k}(t-y_i) + t\prod_{i=1}^{k}(t-y_i) - \sum_{i=1}^{k}c_i\prod_{j\ne i}(t-y_j)
$$
is irreducible in~$\mathbb{Z}[t,y]$.
\end{lem}

\begin{proof}
We regard the polynomial~$p(t,y) \in \mathbb{Z}[t,y] \cong \left(\mathbb{Z}[t]\right)[y]$ as a polynomial of degree two in the variable~$y$, 
with coefficients in~$\mathbb{Z}[t]$. We note that the coefficient of~$y^2$ 
and the constant coefficient~$p(t,0)$ are relatively prime in~$\mathbb{Z}[t]$. 
This follows from the observation that the roots of the coefficient of~$y^2$ are exactly the~$y_i$, 
while none of those numbers is a root of the constant coefficient. Indeed, we have
$$p(y_i,0) = -c_i\prod_{j\ne i}(y_i-y_j) \ne 0.$$
 This implies that the only possibility to factor~$p(t,y)$ is by 
writing it as a product of two factors linear in the variable~$y$. 
To rule this out, we apply Eisenstein's criterion as follows.
The constant coefficient~$p(t,0)$ has a simple root: the Perron-Frobenius eigenvalue of a matrix of the form~$XX^\top$ as in 
Definition~\ref{def:XX}, where we set~$l = g-1 = k$,~$b_i = \sqrt{c_i}$ and~$y=0$. 
Let~$q(t)\in\mathbb{Z}[t]$ be the irreducible factor of~$p(t,0)\in\mathbb{Z}[t]$ containing this root. 
Then~$q(t)$ divides the constant coefficient~$p(t,0)$ but~$q(t)^2$ does not. Furthermore,~$q(t)$ does not divide the 
coefficient of~$y^2$ since otherwise it would have a root in common with the constant coefficient~$p(t,0)$. 
Eisenstein's criterion now implies that~$p(t,y)$ can not be factored into a product of two factors with positive degree in the variable~$y$. 
\end{proof}

\subsection{Main results for strata}
We are now ready to prove the analogues of Theorem~\ref{thm:trace} and Theorem~\ref{thm:stretch}
for strata of Abelian differentials. 

\begin{thm}
\label{thm:all:strata}
Every number~$1\le d\le g$ is realised as the degree of the trace field of a product of two affine multitwists on a surface in every stratum of Abelian differentials on Riemann surfaces of genus~$g$. 
\end{thm}

\begin{proof}
Let~$\mathcal{H}(k_1, k_2, \dots, k_m)$ be a stratum of Abelian differentials. We use the surface constructed in \S\ref{sec:construct}. 
By Thurston--Veech's construction~\cite{Th, Veech:construction}, there exists a flat structure on it, obtained by changing the side lengths of the rectangles, 
such that the multitwists~$T_\alpha$ and~$T_\beta$ have affine representatives, and such that the degree of the trace field is given by the degree of the Perron-Frobenius eigenvalue~$\mu^2$ of~$XX^\top$.

Let~$2 \le d \le g$ be the some degree of a trace field we want to construct. Set~$g-d+1$ of the~$g-1$ parameters~$y_i$ equal to~$1$ and 
all others~$\ne 1$ and pairwise distinct. In this way, the characteristic polynomial of~$XX^\top$ can be factored as~$(t-1)^{g-d}p(t,y)$, where the 
polynomial~$p(t,y)$ is of degree~$d$ in the variable~$t$ and with pairwise distinct~$y_i$. In particular, Lemma~\ref{irred} implies 
that~$p(t,y)$ is irreducible as a polynomial in~$\mathbb{Z}[t,y]$. Now, by Hilbert's irreducibility theorem~\cite{Lang}, there are infinitely many integer specifications of~$y$ such that the resulting polynomial is irreducible in~$\mathbb{Z}[t]$. 
For~$|y|$ large enough, all these specifications can be realised geometrically as in \S\ref{sec:construct}, 
since we start with~$y^2 -g + 1$ squares in the construction. In particular, for every such~$y$, 
we obtain an Abelian differential with trace field of degree~$d$. 
\end{proof}

\begin{thm}
\label{thm:all:strata:stretch}
Every even number~$2\le 2d \le 2g$ is realised as the degree of a product of two affine multitwists on a surface 
in every stratum of Abelian differentials on Riemann surfaces of genus~$g$.
\end{thm}

\begin{proof}
In the proof of Theorem~\ref{thm:all:strata}, we have constructed examples with Perron-Frobenius eigenvalue~$\mu^2$ of~$XX^\top$ 
having degree~$d$ by letting~$g-d+1$ parameters~$y_i$ equal to~$1$. For these examples, we now bound~$\sigma(2I + \Omega)$ 
in order to apply Theorem~\ref{thm:criterion} to~$T_\alpha \circ T_\beta$. Let~$\Omega'$ be the matrix obtained from~$\Omega$ by deleting 
all the rows and all the columns corresponding to~$y$ or the~$g-1 - (g-d+1) = d-2$ parameters~$y_i$ that are not set equal to~$1$. 
We have~$$\sigma(\Omega + 2I) \ge \sigma(\Omega' + 2I) - (d-1).$$ By construction,~$\Omega'$ is the adjacency matrix of a forest 
consisting of path graphs (some of which might be of length zero). In particular, one directly verifies that~$2I + \Omega'$ is positive definite.  
We get~$$\sigma(\Omega + 2I) \ge \sigma(\Omega' + 2I) - (d-1) = n+m - 2d+2 > n+ m - 2d.$$
Furthermore, one directly checks that the matrix~$\sigma(\Omega + 2I)$ has a negative eigenvalue as soon as~$y>4$, which we are allowed to assume. 
This implies~$$n+m>\sigma(\Omega + 2I) + \mathrm{null}(\Omega + 2I).$$ 
Theorem~\ref{thm:criterion} applies and the mapping class~$T_\alpha \circ T_\beta$ is pseudo-Anosov with stretch factor~$\lambda$ of degree~$2d$.
\end{proof}

\begin{remark}
\emph{
The mapping classes we construct above are positive arborescent and so can all obtained by capping off monodromies of certain fibred links called positive arborescent Hopf plumbings. This relation is discussed for example in the background chapter of the second authors' thesis~\cite{Ldiss}. The pseudo-Anosov stretch factors therefore appear as the dominating roots of the Alexander polynomials of these links. It is conceivable that our argument, or at least a portion thereof, could be replaced by a careful analysis of these Alexander polynomials using the skein relation. However, the calculations we present here can readily be applied to our examples in~\S\ref{sec:cc}, which are not necessarily obtained from arborescent Hopf plumbings anymore.
}\end{remark}

\section{Connected components of strata}
\label{sec:cc}

In this section, we study the connected components of strata of Abelian differentials. 
After recalling the classification of the connected components, we first analyse to which 
connected components our examples from \S\ref{sec:stratum} belong. We then 
construct examples covering all remaining connected components, finally proving 
Theorem~\ref{thm:trace} and Theorem~\ref{thm:stretch} in full generality.

\subsection{Classification of connected components of strata}
\label{subsec:cc}

The connected components of the strata of the moduli space of Abelian differentials are classified by~\cite{KZ}. There are at most three connected components, and the classification uses two topological invariants that we describe now.
\begin{enumerate}
\item {\em Hyperellipticity}. For $g\geq 2$, the strata $\mathcal H(2g-2)$ and $\mathcal H(g-1,g-1)$ have a component that consists entirely of hyperelliptic Riemann surfaces, where the hyperelliptic involution permutes the two zeros (when there are two).
\item {\em Parity of the spin structure}. If the degrees of the singularities of a stratum are all even, then one can define a spin structure, or equivalently a quadratic form $q$ on the first homology group. The parity of this spin structure (or the Arf invariant of the form) is a topological invariant.
\end{enumerate}
\begin{remark}
\label{rk:not:hyp}
\emph{
If a translation surface belongs to a hyperelliptic component $\mathcal H^{\mathrm hyp}(2g-2)$ or $\mathcal H^{\mathrm hyp}(g-1,g-1)$ and admits a cylinder decomposition, then all cylinders are fixed by the hyperelliptic involution, and each of them contains exactly two fixed points in its interior. Since the total number of fixed points is $2g+2$, this observation can be used to show that a translation surface does not belong to a hyperelliptic component.}
\end{remark}

We will use the topological definition of the spin structure (see~\cite[\S 3.1]{KZ} for details) to have an effective way to compute its parity in terms of the Arf invariant of $q$. Since the flat metric $(X,\omega)$ has trivial holonomy, outside  of  finite  number  of  singularities, we  have  a  well-defined  horizontal  direction.  Consider  a  smooth  simple closed  oriented  curve~$\alpha$ on~$X$ which does not contain  any singularities. The  total  change  of  the  angle  between  the tangent vector to $\alpha$ and  the  tangent vector to  the  horizontal is  equal  to $2\pi \cdot \mathrm{ind}(\alpha)$, where $\mathrm{ind}(\alpha)\in\mathbb Z$. 
Choose any symplectic basis $(a_i,b_i)_{i=1,\dots,g}$ of $H_1(X;\mathbb Z/2)$. Then the parity of the spin structure is~\cite[Equation~(4)]{KZ}:
\begin{equation}
\label{eq:spin}
\Phi(\omega) = \sum_{i=1}^g q(a_i)q(b_i) \mod 2,
\end{equation}
where $q(\alpha)=\mathrm{ind}(\alpha)+1$ for an oriented smooth path $\alpha$. Together with the formula $q(\alpha+\beta)=q(\alpha)+q(\beta)+i(\alpha,\beta)$ for any $\alpha,\beta\in H_1(X;\mathbb Z/2)$, it is easy to calculate the parity of the spin structure given in any (non symplectic) basis of the first homology.

Next we explain concretely how to compute $\Phi(\omega)$, where $(X,\omega)$ is obtained from the construction in Section~\ref{sec:stratum}.
Observe that $(X,\omega)$  belongs to a non hyperelliptic component if $g>2$. To see this, when $(X,\omega)\in \mathcal H^{hyp}(2g-2)$, note that the number of cylinders we have inserted is $g-1$. By Remark~\ref{rk:not:hyp} they contribute to $2g-2$ fixed points of the hyperelliptic involution (located on the $2g-2$ horizontal core curves), say $p_1,p'_1,\dots,p_{g-1},p'_{g-1}$. There are two more fixed points $q,q'$ on the horizontal core curve of the long cylinder $\mathcal C$ we start with, and one fixed point on its boundary, say $q''$, that is on the same vertical closed curve as $q'$. The last fixed point is the singularity. On the other hand, each inserted cylinder should have two fixed points on its {\em vertical} core curve: one is $p_i$, the other one is $p''_i \in \mathcal C$. Thus necessarily $p''_i=q$ for all $i=1,\dots,g-1$. This is possible only if $g-1=1$.
For $(X,\omega)\in \mathcal H^{hyp}(g-1,g-1)$ the situation is similar.

\subsection{Non hyperelliptic components, spin $1$}
Consider $(X,\omega)$ obtained from the construction in \S\ref{sec:stratum} when all~$k_i$ are even.
As a basis of the first homology $H_1(X,\mathbb Z/2)$, we take horizontal curves~$\gamma_0,\dots,\gamma_{g-1}$ ($\gamma_0$ is the horizontal curve that we start with, 
and~$\gamma_i$ is in the $i$th vertical cylinder), and vertical curves~$\eta_0,\dots,\eta_{g-1}$ ($\eta_0$ crosses $\gamma_0$ only once, and~$\eta_i$ is 
the core curve of the $i$th vertical cylinder for~$i>0$). By construction, for every~$i,j$
$$
i(\gamma_0,\eta_j)=1, \ i(\gamma_i,\eta_j)=\delta_{ij} \textrm{ for } i> 0, \textrm{ and } i(\gamma_i,\gamma_j)=i(\eta_i,\eta_j)=0.
$$
We can thus form a symplectic basis as follows:
$$
\left\{\begin{array}{ll}
a_1 = \gamma_0, & b_1 = \eta_0 \\
a_i = \gamma_{i-1}, &  b_i = \eta_{i-1} - \eta_0 \textrm{ for } i\neq 0
\end{array}
\right.
$$
Clearly $\mathrm{ind}(\gamma_i)=\mathrm{ind}(\eta_i)=0$.
Substituting in Equation~\eqref{eq:spin}, we conclude:
\begin{multline*}
\Phi(\omega) = 1 + \sum_{i=2}^g q(\gamma_{i-1})q(\eta_{i-1}-\eta_0)=\\
=1 + \sum_{i=2}^g (q(\eta_{i-1})+q(\eta_0)+i(\eta_0,\eta_{i-1})) \equiv 1 \mod 2.
\end{multline*}

\subsection{Non hyperelliptic components $\mathcal H(2k_1,\dots,2k_m)$, spin $0$, $m>1$}
\label{sec:non:hyp:spin0}

We now use a slightly different model defined as follows. Start with the surface depicted in Figure~\ref{fig:non:hyp}, with a long horizontal cylinder made of $y^2-g+1$ squares. It belongs to $\mathcal H^{\mathrm{hyp}}(2,2)$.
Its spin structure is $0$ as we can check directly, or by using the formulae in~\cite[Corollary 5]{KZ}. 
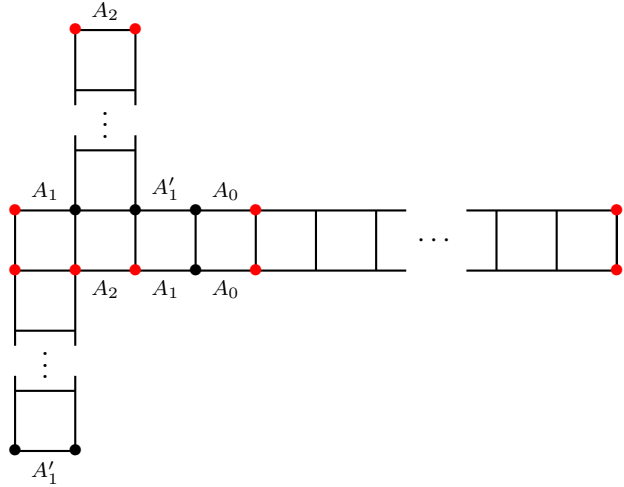
\begin{figure}[htbp]
\begin{tikzpicture}[scale=0.8]
\draw[thick] (0,0) -- (0,1) -- (4,1) -- (6.5,1);   \draw[thick] (7.5,1) -- (10,1)  -- (10,0) -- (7.5,0);
\draw[thick] (0,0) -- (6.5,0);
\draw[thick] (1,1) -- (1,2.25); \draw[thick] (1,2.75)--(1,4)--(2,4)--(2,2.75);
\draw[thick] (2,2.25)--(2,1);
\draw[thick] (1,2)--(2,2);\draw[thick] (1,3)--(2,3);

\draw[thick] (0,0) -- (0,-1.25); 
\draw[thick] (0,-1.75) -- (0,-3) -- (1,-3) -- (1,-1.75);
\draw[thick] (1,-1.25)--(1,0); 
\draw[thick] (0,-1)--(1,-1); \draw[thick] (0,-2)--(1,-2); 

\foreach \i in {1,2,3,4,5,6,8,9,10}
{\draw[thick] (\i,0) -- (\i,1);
}
\foreach \p in {(0,0), (1,0), (2,0),(4,0),(4,1),(1,4),(2,4),(0,1),(10,0),(10,1)}
{
\draw[red] \p node{$\bullet$};
}
\foreach \p in {(1,1), (2,1), (3,1),(3,0),(0,-3),(1,-3)}
{
\draw[black] \p node{$\bullet$};
}
\draw (0.5,-3) node[below]{$\scriptstyle A'_1$}; \draw (2.5,1) node[above]{$\scriptstyle A'_1$};
\draw (0.5,1) node[above]{$\scriptstyle A_1$}; \draw (2.5,0) node[below]{$\scriptstyle A_1$};
\draw (1.5,4) node[above]{$\scriptstyle A_2$}; \draw (1.5,0) node[below]{$\scriptstyle A_2$};
\draw (3.5,1) node[above]{$\scriptstyle A_0$}; \draw (3.5,0) node[below]{$\scriptstyle A_0$};
\draw (0.5,-0.9) node[below]{$\vdots$}; \draw (1.5,3.1) node[below]{$\vdots$};

\draw (7,0.5) node {$\cdots$};
\end{tikzpicture}
\caption{A surface in $\mathcal H^{\mathrm{hyp}}(2,2)$ (with even spin structure).}
\label{fig:non:hyp}
\end{figure}
We can insert~$g-3\geq 1$ vertical strips of~$y_i + 1$ square tiles (for~$i=3,\dots,g-1$) as in Section~\ref{sec:stratum} in order to add zeros of even multiplicities and to reach the stratum $\mathcal H(2k_1,\dots,2k_m)$ where $\sum 2k_i=2g-2$.
This construction does not change the spin structure as we can see on the computation below. We let~$\gamma_0$ the horizontal core curve in the long cylinder, and~$\gamma_1,\dots,\gamma_{g-1}$ the other horizontal core curves contained in the~$i$th cylinder. Similarly, we let~$\eta_i$ for~$i=0,\dots,g-1$ the vertical core curves:~$\eta_0$ is the core curve of the vertical cylinder with label~$A_0$ and~$\eta_i$ is the core curve of the~$i$th vertical cylinder for~$i>0$. We have for every~$i,j$
$$
\begin{array}{l}
i(\gamma_0,\eta_i)=1 \textrm{ for } i \neq 1 \textrm{ and } i(\gamma_0,\eta_1)=2, \\
i(\gamma_i,\eta_i)=\delta_{ij} \textrm{ for } i > 0, \\
i(\gamma_i,\gamma_j)=i(\eta_i,\eta_j)=0.
\end{array}
$$
We can thus form a symplectic basis of $H_1(S;\mathbb Z_2)$ as follows:
$$
\left\{\begin{array}{ll}
a_1 = \gamma_0, & b_1 = \eta_0 \\
a_2 = \gamma_1, & b_2 = \eta_1 \\
a_i = \gamma_{i-1}, &  b_i = \eta_{i-1} - \eta_0 \textrm{ for } i>2
\end{array}
\right.
$$
By using Equation~\eqref{eq:spin} this leads to
\begin{multline*}
\Phi(\omega) = 1 + 1 + \sum_{i=3}^g q(\gamma_{i-1})q(\eta_{i-1}-\eta_0)=\\
=1 + 1 + \sum_{i=3}^g (q(\eta_{i-1})+q(\eta_0)+i(\eta_0,\eta_{i-1})) \equiv 0 \mod 2.
\end{multline*}
We now compute the degree of the trace field.
In order to write down the matrix~$XX^\top$, we apply the strategy described in \S\ref{sec:calculating}. Observe that the horizontal curve that we start with crosses $y^2-g+1+g-3=y^2-2$ squares.
More precisely it intersects $y^2-4$ vertical curves once and one vertical curve twice. 
We obtain the following  matrix, where~$\mathbf{b}_{n\times m}$ stands for the~$n\times m$ matrix with all entries equal to~$b\in \mathbb Z$:
\[
XX^\top=
\left(
\begin{array}{c|cccc}
y^2 & \mathbf{2}_{1\times y_{1}} &  \mathbf{1}_{1\times y_{2}} & \cdots & \mathbf{1}_{1\times y_{g-1}} \\
\hline
\mathbf{2}_{y_{1}\times 1} & \mathbf{1}_{y_1\times y_{1}} & &&\\
\mathbf{1}_{y_{2}\times 1} & & \mathbf{1}_{y_2\times y_{2}}& &\\
\vdots & & & \ddots &\\
\mathbf{1}_{y_{g-1}\times 1} & & & & \mathbf{1}_{y_{g-1}\times y_{g-1}}\\
\end{array}
\right).
\]
From Lemma~\ref{charpoly} with~$l=0$, we see that the characteristic polynomial of~$XX^\top$ equals~$t^a\cdot p(t,y,\mathbf{y}),$ where
$$
p(t,y,\mathbf{y})=
-y^2\prod_{i=1}^{g-1}(t-y_i) + t\prod_{i=1}^{g-1}(t-y_i) - \sum_{i=1}^{g-1} c_i\prod_{j\ne i}(t-y_j),
$$
for~$a=\sum_{i=1}^{g-1}(y_i-1)$, $c_{1}= 4y_{1}$ and $c_i=y_i$ if~$i\ge 2$.
From Lemma~\ref{irred}, we deduce that~$p(t,y,\mathbf{y})$ is irreducible in~$\mathbb{Z}[t,y]$ given that all~$y_i\in \mathbb{N}$ are pairwise distinct (here our parameter $b$ in Definition~\ref{def:XX} equals~$2$). 
As before, we can factor out~$(t-1)^{g-d}$ and obtain an irreducible polynomial of degree~$d$ by setting~$g-d+1$ of the~$g-1$ parameters~$y_i$ equal to~$1$. 
We can then apply the same strategy than the proof of Theorem~\ref{thm:all:strata} to get the result.
\begin{cor}
Every number~$1\le d\le g$ is realised as the degree of the trace field of a product of two affine multitwists on a surface in every non hyperelliptic connected component with spin $0$ of a stratum of Abelian differentials on Riemann surfaces of genus~$g$. 
\end{cor}
We further apply the same strategy to realise all even degrees as stretch factors. 
One can copy the proof of Theorem~\ref{thm:all:strata:stretch} word for word and obtain the following result.
\begin{cor}
Every even number~$2\le 2d\le 2g$ is realised as the degree of a product of two affine multitwists on a surface in every non hyperelliptic 
connected component with spin $0$ of a stratum of Abelian differentials on Riemann surfaces of genus~$g$. 
\end{cor}

\subsection{Reaching non hyperelliptic component of $\mathcal H(2g-2)$, spin $0$.}

\subsubsection{Degree $d=2$}

We start with the model presented in Figure~\ref{fig:non:hyp}, and insert~$g-3$ vertical cylinders~($g>3$) with parameters~$y_1=2$ and~$y_i=1$ for~$i>1$ (see also Figure~\ref{fig:modif}).
The number of squares in grey color is $y^2-2-3-(g-3)=y^2-g-2$. The surface belongs to $\mathcal H(2,2g-4)$. Since it can be continuously deformed to the surface in Figure~\ref{fig:non:hyp} with spin $0$, it also has spin $0$. Now we collapse all the grey squares. The resulting surface belongs to the stratum $\mathcal H^{\textrm{non hyp}}(2g-2)$. Again this continuous deformation does not change the parity of the spin structure.

\begin{figure}[htbp]
\begin{tikzpicture}[scale=0.8]
\draw[thick] (0,0) -- (0,1) -- (4,1) -- (4,1);   \draw[thick] (5,1) -- (10,1)  -- (10,0) -- (5,0);
\draw[thick] (0,0) -- (4,0);
\draw[thick] (1,1) -- (1,2) -- (2,2)--(2,1);
\draw[thick] (6,0) -- (6,-1) -- (7,-1) -- (7,0);
\draw[thick] (9,1) -- (9,2) -- (10,2) -- (10,1);
\draw[thick] (1,2)--(2,2);

\draw[fill=gray!80] (3,0) -- (3,1) -- (4,1) -- (4,0) -- cycle; 
\draw[fill=gray!80] (5,0) -- (5,1) -- (6,1) -- (6,0) -- cycle; 

\foreach \p in {(3,0),(1,1),(2,1),(0,-2),(1,-2),(3,1)}
{
\draw[red] \p node{$\bullet$};
}

\foreach \p in {(0,0),(1,2),(2,2),(6,1),(7,1),(9,1),(10,1),(6,0),(7,0),(9,0),(10,0),(10,2),(9,2),(6,-1),(7,-1),(0,1),(1,0),(2,0)}
{
\draw[black] \p node{$\bullet$};
}

\draw[thick] (0,0) -- (0,-1.25); 
\draw[thick] (0,-2) -- (0,-1) -- (1,-1) -- (1,-2);
\draw[thick] (1,-1.25)--(1,0); 
\draw[thick] (0,-1)--(1,-1); \draw[thick] (0,-2)--(1,-2); 

\foreach \i in {1,2,3,4,5,7,9,6}
{\draw[thick] (\i,0) -- (\i,1);
}
\draw (0.5,-2) node[below]{$\scriptstyle A'_1$}; \draw (2.5,1) node[above]{$\scriptstyle A'_1$};
\draw (0.5,1) node[above]{$\scriptstyle A_1$}; \draw (2.5,0) node[below]{$\scriptstyle A_1$};
\draw (1.5,2) node[above]{$\scriptstyle A_2$}; \draw (1.5,0) node[below]{$\scriptstyle A_2$};
\draw (3.5,1) node[above]{$\scriptstyle A_0$}; \draw (3.5,0) node[below]{$\scriptstyle A_0$};
\draw (6.5,1) node[above]{$\scriptstyle A_3$}; \draw (6.5,-1) node[below]{$\scriptstyle A_3$};
\draw (9.5,2) node[above]{$\scriptstyle A_{g-1}$}; \draw (9.5,0) node[below]{$\scriptstyle A_{g-1}$};
\draw (8,0.5) node {$\cdots$};
\draw (4.5,0.5) node {$\cdots$};
\end{tikzpicture}
\caption{
\label{fig:modif}
A surface in $\mathcal H^{\textrm{non hyp}}(2,2g-4)$ for $g>3$. If we collapse the handle (in grey color) we obtain a surface
in $\mathcal H^{\textrm{non hyp}}(2g-2)$.}
\end{figure}
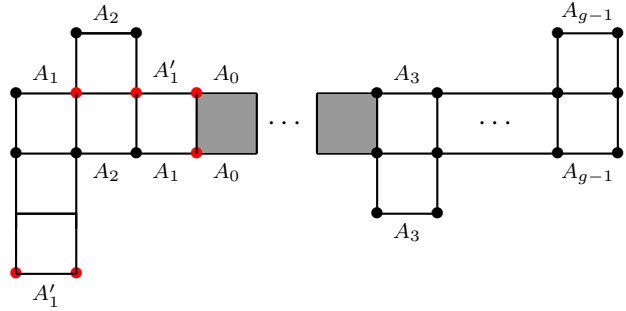

Following the computation in the previous subsection, we now obtain the $(g+1)\times (g+1)$ intersection matrix (recall $y^2-g-2=0$)
\[
XX^\top=
\left(
\begin{array}{c|ccccc}
g+2 & 2 & 2 & 1 & \cdots & 1 \\
\hline
2 & 1 &1 &&\\
2 & 1& 1& &\\
1 & & & 1& &\\
\vdots & & & & \ddots &\\
1 & & & & & 1\\
\end{array}
\right).
\]
By Lemma~\ref{charpoly}, the characteristic polynomial of~$XX^\top$ equals
$$
t^a\left(-(g+2)\prod_{i=1}^{g-1}(t-y_i) + t\prod_{i=1}^{g-1}(t-y_i) - \sum_{i=1}^{g-1} c_i\prod_{j\ne i}(t-y_j)   \right),
$$
where $a=1$, $c_1=4y_1$ and $c_i=y_i$ for $i\geq2$. Thus the polynomial is
\begin{multline*}
t(t-1)^{g-3} ( -(g+2)(t-2)(t-1) + t(t-2)(t-1)  -8(t-1) -(g-2)(t-2)    ) = \\
= t^2(t-1)^{g-3} (t^2-t\cdot(g+5)+2g+2 ).
\end{multline*}
In particular, the degree of the trace field is either one or two. 
The discriminant of $t^2-t\cdot(g+5)+2g+2$ is $D=(g+5)^2-8\cdot(g+1)=g^2+2g+17$. 
We see that $(g+1)^2 < D < (g+5)^2$. 
If the degree of the trace field is one then $D$ is a square, and one of the following three cases holds:
\begin{enumerate}
\item $D=(g+2)^2$. Then $g^2+2g+17=g^2+4g+4$. Solving in $g$ we find $2g=13$ which is a contradiction.
\item $D=(g+3)^2$. Then $g^2+2g+17=g^2+6g+9$. Solving in $g$ we find $6g=1$ which is a contradiction.
\item $D=(g+4)^2$. Then $g^2+2g+17=g^2+4g+4$. Solving in $g$ we find $g=2$ which is again a contradiction with $g>3$.
\end{enumerate}
This implies that~$D$ is not a square and hence the degree of the trace field must be two.

\subsubsection{Degree $2<d\leq g$}

We consider the modified version of our construction as depicted in Figure~\ref{fig:modified}. When $g>3$ the surface is not hyperelliptic.
\begin{figure}[htbp]
\begin{tikzpicture}[scale=0.8]

\draw[thick] (-1.25,0) -- (0,0) -- (0,1) -- (3.25,1);  \draw[thick] (-1.25,1) -- (0,1);

\foreach \i in {-2,-1,1,2,3,4,5}
{\draw[thick] (\i,0) -- (\i,1);}

\draw[thick] (0,0) -- (3.25,0);
\draw[thick] (3.75,0) -- (6,0) -- (6,1) -- (3.75,1);
\draw[thick] (5,1) -- (5,2) -- (6,2) -- (6,3) -- (7,3) -- (7,1) -- (6,1); \draw[thick] (6,1) -- (6,2);\draw[thick] (6,2) -- (7,2);

\draw[thick] (-1.75,0) -- (-3,0) -- (-3,1) -- (-1.75,1);

\draw[thick] (1,1) -- (1,2.25); \draw[thick] (1,2.75)--(1,4)--(2,4)--(2,2.75);
\draw[thick] (2,2.25)--(2,1);
\draw[thick] (1,2)--(2,2);\draw[thick] (1,3)--(2,3);

\draw[thick] (2,0) -- (2,-1.25); \draw[thick] (2,-1.75)--(2,-3)--(3,-3)--(3,-1.75);
\draw[thick] (3,-1.25)--(3,0);
\draw[thick] (2,-1)--(3,-1);\draw[thick] (2,-2)--(3,-2);

\draw[thick] (4,0) -- (4,-1.25); \draw[thick] (4,-1.75)--(4,-3)--(5,-3)--(5,-1.75);
\draw[thick] (5,-1.25)--(5,0);
\draw[thick] (4,-1)--(5,-1);\draw[thick] (4,-2)--(5,-2);

\foreach \i in {-3,1,2,3,4,5,6}
{\draw[black] (\i,0) node{$\bullet$};
\draw[black] (\i,1) node{$\bullet$};
}
\draw[black] (7,1) node{$\bullet$};
\draw[black] (5,2) node{$\bullet$};\draw[black] (6,2) node{$\bullet$};\draw[black] (7,2) node{$\bullet$};
\draw[black] (6,3) node{$\bullet$};\draw[black] (7,3) node{$\bullet$};
\draw[black] (1,4) node{$\bullet$};\draw[black] (2,4) node{$\bullet$};
\draw[black] (2,-3) node{$\bullet$};\draw[black] (3,-3) node{$\bullet$};
\draw[black] (4,-3) node{$\bullet$};\draw[black] (5,-3) node{$\bullet$};

\draw[red] (-1,0) node{$\bullet$};\draw[red] (-1,1) node{$\bullet$};
\draw[red] (-2,0) node{$\bullet$};\draw[red] (-2,1) node{$\bullet$};
\draw[red] (0,0) node{$\bullet$};\draw[red] (0,1) node{$\bullet$};
\draw (-1.45,0.5) node {$\dots$};  \draw (3.51,0.5) node {$\dots$};
\draw (1.5,2.6) node {$\vdots$}; \draw (2.5,-1.4) node {$\vdots$}; \draw (4.5,-1.4) node {$\vdots$};

\draw (-1.5,1) node[above]{$\scriptstyle A_0$}; \draw (-1.5,0) node[below]{$\scriptstyle A_0$};
\draw (1.5,0) node[below]{$\scriptstyle A_1$}; \draw (1.5,4) node[above]{$\scriptstyle A_1$};
\draw (2.5,1) node[above]{$\scriptstyle A_2$}; \draw (2.5,-3) node[below]{$\scriptstyle A_2$};
\draw (4.5,1) node[above]{$\scriptstyle A_{g-3}$}; \draw (4.5,-3) node[below]{$\scriptstyle A_{g-3}$};
\draw (5.5,2) node[above]{$\scriptstyle A_{g-2}$}; \draw (5.5,0) node[below]{$\scriptstyle A_{g-2}$};
\draw (6.5,3) node[above]{$\scriptstyle A_{g-1}$}; \draw (6.5,1) node[below]{$\scriptstyle A_{g-1}$};

\end{tikzpicture}
\caption{\label{fig:modified}
A surface in the even spin non hyperelliptic connected component of $\mathcal H(2g-2)$ for $g>3$.}
\end{figure}
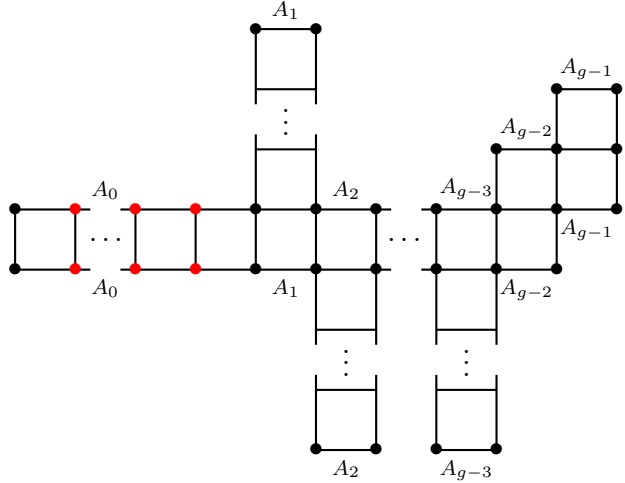

For the computation of the spin structure, we consider the ``obvious'' core curves $\gamma_i$ and $\eta_i$ (for $i=0,\dots,g-1$) of the horizontal and vertical cylinders.
It forms a (non symplectic) basis of the homology:
$$
\begin{array}{l}
i(\gamma_0,\eta_i)=1 \textrm{ for } i \neq g-1 \textrm{ and } i(\gamma_0,\eta_{g-1})=0, \\
i(\gamma_i,\eta_i)=\delta_{ij} \textrm{ for } i= 0,\dots,g-1 \\
i(\gamma_i,\gamma_j)=i(\eta_i,\eta_j)=0.
\end{array}
$$
We can thus form a symplectic basis of $H_1(S;\mathbb Z_2)$ as follows:
$$
\left\{\begin{array}{ll}
a_1 = \gamma_0, & b_1 = \eta_0 \\
a_i = \gamma_{i-1}, &  b_i = \eta_{i-1} - \eta_0\  \textrm{ for } i=2,\dots,g-1 \\
a_{g} = \gamma_{g-1}, &  b_g = \eta_{g-1} - b_{g-1} = \eta_{g-1} - \eta_{g-2} + \eta_0
\end{array}
\right.
$$
Equation~\eqref{eq:spin} reads
$$
\Phi(\omega) = q(\gamma_0)q(\eta_0) + \sum_{i=2}^{g-1} q(\gamma_{i-1})q(\eta_{i-1}-\eta_0) + q(\gamma_{g-1})q(\eta_{g-1} - \eta_{g-2} + \eta_0)
$$
Since $q(\eta_{i-1}-\eta_0) = q(\eta_{i-1}) + q(\eta_0) + i(\eta_{i-1},\eta_0) = 1 + 1 + 0 = 0\mod 2$, the sum with the $g-2$ terms vanishes.
For the last term, a direct computation leads to
\begin{multline*}
q(\eta_{g-1} - \eta_{g-2} + \eta_0) = q(\eta_{g-1}) + q(\eta_{g-2}) + q(\eta_0) + \\ 
+ i(\eta_{g-1},\eta_{g-2}) + i(\eta_{g-1},\eta_0) + i(\eta_{g-2} + \eta_0) = \\
= 1 + 1 + 1 + 0 + 0 + 0 = 1 \mod 2.
\end{multline*}
Finally we get $\Phi(\omega) = 1 + 0 + 1 = 0 \mod 2$.

The intersection matrix (with the parameters $y_{g-2}=y_{g-1}=1$) is
\[
XX^\top=
\left(
\begin{array}{c|ccccc}
y^2 & \mathbf{1}_{1\times y_{1}} & \cdots & \mathbf{1}_{1\times y_{g-3}} & 1 & 0\\
\hline
\mathbf{1}_{y_{1}\times 1} & \mathbf{1}_{y_1} & &&0&0\\
\vdots & & \ddots &&\vdots&\vdots\\
\mathbf{1}_{y_{g-3}\times 1} & & & \mathbf{1}_{y_{g-3}}&0&0\\
1 &0 & \cdots& 0&2&1\\
0 &0 &\cdots & 0&1&1\\
\end{array}
\right).
\]
By developing along the last column, its characteristic polynomial equals
\begin{multline}
(t-1)\left( (t-2)p(t,y,\mathbf{y})- t^a\prod_{i=1}^{g-3}(t-y_i) \right) - p(t,y,\mathbf{y})=\\
p(t,y,\mathbf{y})(t^2-3t+1)-t^a(t-1)\prod_{i=1}^{g-3}(t-y_i)=\\
= t^a\left(-y^2(t^2-3t+1)\prod_{i=1}^{g-3}(t-y_i) + (t^3-3t^2+1)\prod_{i=1}^{g-3}(t-y_i) \right. \\ \left. -(t^2-3t+1)\sum_{i=1}^{g-3} c_i\prod_{j\ne i}(t-y_j)   \right),
\end{multline}
where $p(t,y,\mathbf{y})$ is the degree $g-2$ polynomial in Lemma~\ref{charpoly}, 
with the parameters~$a=\sum_{i=1}^{g-3}(y_i-1)$ and $c_i=y_i$ for all~$1\le i \le g-3$.
Following the same line of proof we used for Lemma~\ref{irred}, we show
\begin{lem}
The polynomial
$$
-y^2(t^2-3t+1)\prod_{i=1}^{g-3}(t-y_i) + (t^3-3t^2+1)\prod_{i=1}^{g-3}(t-y_i) - (t^2-3t+1)\sum_{i=1}^{g-3} y_i\prod_{j\ne i}(t-y_j)   
$$
is irreducible in~$\mathbb{Z}[t,y]$ given that all~$y_i\in \mathbb{N}$ are distinct.
\end{lem}
\begin{proof}
We follow the proof of Lemma~\ref{irred}.
We note that the polynomial has degree two in the variable~$y$ with no non trivial common factor between the coefficient of~$y^2$ and the constant 
coefficient. By the Perron-Frobenius theorem, there is a simple irreducible factor of the constant coefficient. Thus Eisenstein's criterion applies in $\left(\mathbb{Z}[t]\right)[y]$.
\end{proof}

\subsection{Reaching the hyperelliptic components of strata~$\mathcal H(2g-2)$ and $\mathcal H(g-1,g-1)$}

We start by constructing a square-tiled surface. Pick a long horizontal square-tiled cylinder made of $2n+1$ squares with 
identifications $A_0$, $A_1,\dots,A_n$ and $A'_1,\dots,A'_n$ as depicted in Figure~\ref{hyp}. We then add a stair case 
template, made of $k$ steps, using a total of $2k$ squares. Finally we insert a long vertical square-tiled cylinder with some large 
number~$y$ of squares and identifications $C_1,\dots,C_y$ as in Figure~\ref{hyp}. We treat~$y$ as a variable that we will need to specify later on. 
This creates a surface $X_{n,k,y}$. Similarly, one can construct a surface $Y_{n,k,y}$ be collapsing one square corresponding to the label $A_0$. 
\begin{figure}[htbp]
\begin{tikzpicture}[scale=0.8]
\draw[thick] (0,0) -- (0,1) -- (1.25,1); \draw[thick] (1.75,1) -- (5.25,1);
\draw[thick] (5.75,1) -- (8,1) -- (8,0) -- (5.75,0);
\draw[thick] (5.25,0) --  (1.75,0);
\draw[thick] (0,0) --  (1.25,0);
\foreach \i in {1,2,3,4,5,6,7}
{\draw[thick] (\i,0) -- (\i,1);}
\filldraw [draw=black, fill=gray!80] (3,0) -- (3,1) -- (4,1) -- (4,0);

\foreach \i in {0,1}
{\draw[thick] (7+\i ,1+\i) -- (7+\i,2+\i) -- (9+\i,2+\i) -- (9+\i,1+\i) -- cycle;
\draw[thick] (8+\i ,1+\i)--(8+\i ,2+\i);}

\draw[thick] (10,4) -- (10,5) -- (12,5) -- (12,4) -- cycle;
\draw[thick] (11,4) -- (11,7.25);
\draw[thick] (11,7.75)  -- (11,10) -- (12,10) -- (12,7.75);
\draw[thick] (12,5) -- (12,7.25);

\foreach \i in {6,7,8,9}
{\draw[thick] (11,\i) -- (12,\i);}

\foreach \p in {(3,1),(3,0),(5,1),(1,0),(7,1),(9,1),(9,3),(10,4),(12,4),(12,10),(11,6),(12,5),(10,5),(9,2),(7,2),(7,0),(1,1),(5,0)}
{
\draw[fill=black] \p circle (2.2pt);
}
\foreach \p in {(0,1),(6,0),(2,1),(4,0),(4,1),(2,0),(6,1),(0,0),(8,0),(8,2),(10,2),(11,5),(12,9),(11,10),(11,4),(10,3),(8,3),(8,1)}
{
\draw[fill=white] \p circle (2.2pt);
}

 \draw (2.5,1) node[above]{$\scriptstyle A_n$}; \draw (0.5,1) node[above]{$\scriptstyle A_1$};
 \draw (2.5,0) node[below]{$\scriptstyle A'_1$}; \draw (0.5,0) node[below]{$\scriptstyle A'_n$};
\draw (1.5,0.5) node {$\dots$};
\draw (3.5,1) node[above]{$\scriptstyle A_0$}; \draw (3.5,0) node[below]{$\scriptstyle A_0$};
\draw (4.5,1) node[above]{$\scriptstyle A'_1$}; \draw (4.5,0) node[below]{$\scriptstyle A_n$};
\draw (6.5,1) node[above]{$\scriptstyle A'_n$}; \draw (6.5,0) node[below]{$\scriptstyle A_1$};
\draw (5.5,0.5) node {$\dots$};
\draw (7.5,2) node[above]{$\scriptstyle B_1$}; \draw (7.5,0) node[below]{$\scriptstyle B_1$};
\draw (8.5,3) node[above]{$\scriptstyle B_2$}; \draw (8.5,1) node[below]{$\scriptstyle B_2$};
\draw (9.5,2) node[below]{$\scriptstyle B_3$}; \draw (10.1,5) node[above]{$\scriptstyle B_k$};

\draw (12,5.5) node[right]{$\scriptstyle C_1$};\draw (12,6.5) node[right]{$\scriptstyle C_2$};
\draw (12.1,7.5) node[right]{$\vdots$};
\draw (12,9.5) node[right]{$\scriptstyle C_y$};\draw (12,8.5) node[right]{$\scriptstyle C_{y-1}$};
\draw (11,5.5) node[left]{$\scriptstyle C_y$};\draw (11,6.5) node[left]{$\scriptstyle C_1$};
\draw (11,7.5) node[left]{$\scriptstyle C_2$}; \draw (10.9,8.5) node[left]{$\vdots$}; 
\draw (11,9.5) node[left]{$\scriptstyle C_{y-1}$};

\draw (10,3.1) node[above]{$\iddots$}; 
\draw (11.5,7.1) node[above]{$\vdots$}; 

\end{tikzpicture}
\caption{The surface $X_{n,k,y} \in \mathcal H(g-1,g-1)$ made of $2n+2k+y+2$ squares, and the surface $Y_{n,k,y} \in \mathcal H(2g-2)$ obtained from $X_{n,k,y}$ by collapsing the grey square.}
\label{hyp}
\end{figure}
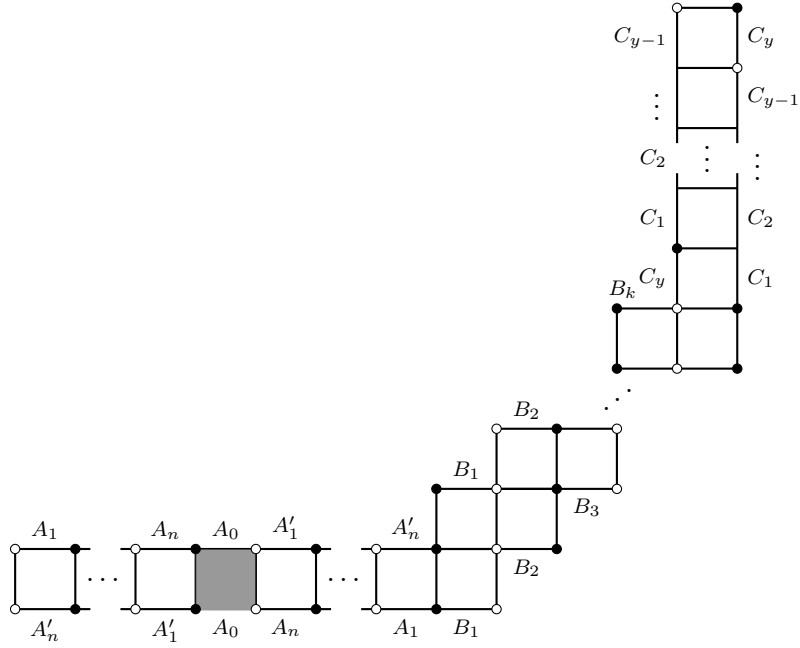

\begin{lem}
\label{hyp:strata}
The genus of $X_{n,k,y}$ and $Y_{n,k,y}$ is $g=n+k+2$. Moreover $X_{n,k,y}$ belongs to the hyperelliptic connected component of $\mathcal H(g-1,g-1)$ while $Y_{n,k,y}$ belongs to the hyperelliptic connected component of $\mathcal H(2g-2)$.
\end{lem}

\begin{proof}[Proof of Lemma~\ref{hyp:strata}]
Clearly the two square-tiled surfaces are hyperelliptic: the involution fixes the $k+2$ horizontal cylinders. By inspecting the gluing, one sees that $X_{n,k,y}$ has two  zeros, each of order $g-1$. The cone angle at each zero is $g\cdot 2\pi$. Since the total number of squares contributing to the cone angle is $2n+2k+2+2$, we get $(2n+2k+4)\cdot 2\pi=g\cdot 2\pi + g\cdot 2\pi$. Hence, $g=n+k+2$.\\
Similarly, $Y_{n,k,y}$ has one zero, of order $2g-2$ and cone angle $(2g-1)\cdot 2\pi$. Now the total number of squares contributing to the cone angle is one less: $2n+2k+2+1$. Thus, $(2n+2k+3)\cdot 2\pi=(2g-1)\cdot 2\pi$.
\end{proof}
\noindent A quick inspection of the intersections of horizontal curves with vertical curves yields that~$XX^\top$ is the following $(k+2)\times (k+2)$ Jacobi matrix:
\[
XX^\top=
\left(
\begin{array}{c|cccccc}
\alpha & 1 & \\
\hline 
1 & 2 & 1 &  \\
 & 1 & \ddots & \ddots &  \\
 &  & \ddots & \ddots & 1 \\
  &  &  & 1 & 2 & y \\
  & &  &  & y & y^2   \\
\end{array}
\right),
\]
where $\alpha=2+4n$ if one considers $X_{n,k,y}$, and  $\alpha=1+4n$ otherwise.
\begin{lem}
\label{charpoly:XXt:hyp}
If $\alpha\neq 1$, then the characteristic polynomial of~$XX^\top$, when regarded as a polynomial in the variables $y$ and $t$, is irreducible in $\mathbb Z[t,y]$.
\end{lem}
\begin{proof}[Proof of Lemma~\ref{charpoly:XXt:hyp}]
Let $p_k$ be the characteristic polynomial of $XX^\top$. We will use the characteristic polynomial $q_k(t)$ of the $(k+1)\times(k+1)$ matrix $B_k$ obtained from $2I_{k+1} + \mathrm{Ad}(A_{k+1})$ by adding~$\alpha-2$ to the first diagonal entry, where $\mathrm{Adj}(A_k)$ is the adjacency matrix of the path graph with $k$ vertices.
We obtain directly by developing the determinant of $t\mathrm{Id}_{k+2}- XX^\top$ along the last column that $p_k(t,y)=-y^2(q_{k}(t)+q_{k-1}(t)) + tq_{k}(t)$. 
We now claim that the roots of $q_{k}$ and $q_{k-1}$ are pairwise distinct and simple.

\begin{proof}[Proof of the claim]
We note that~$B_{k-1}$ is obtained from~$B_{k}$ by deleting the last row and the last column. 
Interlacing results for real symmetric matrices tell us that the eigenvalues of~$B_{k}$ and~$B_{k-1}$ interlace. 
This means that if~$\lambda_1 \le \cdots \le \lambda_{k+1}$ are the eigenvalues of~$B_{k}$ and if~$\mu_1 \le \cdots \le \mu_{k}$ 
are the eigenvalues~$B_{k-1}$, then we have~$$\lambda_i \le \mu_i \le \lambda_{i+1}$$ for all~$1\le i \le {k}$. 
The crucial point is that in our case these inequalities are strict, which can be proved as follows.
We first note that the matrix~$B_k$ is clearly symmetric and positive definite (for~$\alpha \ge 1$). 
This implies that all its leading principal minors are positive.
Now, Theorem~7 by Gantmacher and Krein~\cite{GK} states that a tridiagonal matrix with positive coefficients 
on the main diagonal and the adjacent diagonals is oscillatory if and only if all the leading principal minors are positive, 
implying that~$B_k$ is oscillatory. 
In turn, Theorem~6.5 by Ando~\cite{Ando} states that for oscillatory matrices, all the interlacing inequalities are strict. 
That is, if~$\lambda_1 \le \cdots \le \lambda_{k+1}$ are the eigenvalues of~$B_{k}$ and if~$\mu_1 \le \cdots \le \mu_{k}$ 
are the eigenvalues~$B_{k-1}$, then we have
$$\lambda_i < \mu_i < \lambda_{i+1}$$ for all~$1\le i \le {k}$. 
In particular, the eigenvalues of~$B_{k}$ and the eigenvalues of~$B_{k-1}$ are pairwise distinct and simple.
\end{proof}
We now finish the proof the lemma. Let $F\neq t$ be an irreducible factor of~$q_k$. Since the roots of~$q_k$ are simple, $F^2$ is not a factor of $tq_k$. If $F$ is a factor of 
$q_k+q_{k-1}$ then $q_k$ and $q_{k-1}$ share a common root, which is not possible by the claim. Hence, by Eisenstein's criterion, $p_{k}(t,y)$ is irreducible when regarded as a polynomial in the variable~$y$ and so can not be factored in the form $(ay+b)(cy+d)$. So, if there is a factorisation of $p_{k}(t,y)$, then one of the factors must have degree zero in the variable $y$. But such a factorisation cannot exist, since $q_k(t)+q_{k-1}(t)$ and $tq_k(t)$ are relatively prime in $\mathbb Z[t]$. Indeed, since the roots of $q_k(t)$ and $q_{k-1}(t)$ are distinct, the only possible common factor of~$tq_k$ and~$q_k + q_{k-1}$ is $t$. But $p_k(0,y)$ is the determinant of~$XX^\top$ and equals~$y^2\cdot(\alpha-1)\neq 0$. This proves the lemma.
\end{proof}
\begin{thm}
\label{thm:hyp}
For any hyperelliptic connected component $\mathcal C$ of $\mathcal H(2g-2)$,
every number $1 \leq d \leq g-1$ is realised as the degree of the trace field of a product of two affine multitwists on a surface in $\mathcal C$.
\end{thm}

\begin{proof}[Proof of Theorem~\ref{thm:hyp}]
Since the case $d=1$ is clear by considering square-tiled surfaces, let us assume $d\geq 2$ and set $k=d-2\geq 0$. We construct a surface $X_{n,k,y}$ or $Y_{n,k,y}$ where $n=g-d=g-k-2\geq 0$ (see Lemma~\ref{hyp:strata}). If $d<g$ then $n\neq 0$ and $\alpha \neq 1$. If $d=g$, that is, $n=0$, then by assumption we consider only $Y_{n,k,y} \in \mathcal H(g-1,g-1)$ so that $\alpha = 4n+2=2 \neq 1$. Thus Lemma~\ref{charpoly:XXt:hyp} applies and the characteristic polynomial of~$XX^\top$, viewed as a polynomial in $\mathbb Z[y,t]$ is irreducible. Then by Hilbert's irreducibility theorem, there are infinitely many specifications of $y$ so that the resulting polynomial is irreducible as a polynomial in the variable $t$. Note that all specifications can be realised geometrically. Indeed, one can choose $y>0$ by symmetry. In particular, applying Thurston--Veech's construction, there exists a product of two multitwists on the surface of genus $g = n+d$ in the desired connected component.
\end{proof}
We also prove the analogous theorem for degrees of stretch factors. 
\begin{thm}
\label{thm:hyp:stretch}
For any hyperelliptic connected component $\mathcal C$ of $\mathcal H(g-1,g-1)$,
every even number $2 \leq 2d \leq 2g$ is realised as the degree of the stretch factor of a product of two affine multitwists on a surface in $\mathcal C$.

For any hyperelliptic connected component $\mathcal C$ of $\mathcal H(2g-2)$,
every even number $2 \leq 2d \leq 2g-2$ is realised as the degree of the stretch factor of a product of two affine multitwists on a surface in $\mathcal C$.
\end{thm}

\begin{proof}
We use the same examples as in the proof of Theorem~\ref{thm:hyp}.
We first deal with the case~$d = 2$ by taking the specific example~$y=1$. 
In this case, we have~$k=0$ and we get~$XX^\top = \begin{pmatrix} \alpha & 1 \\ 1 & 1 \end{pmatrix}$. 
We obtain
$$\mu^2 = \frac{\alpha + 1 + \sqrt{\alpha^2 - 2\alpha + 5}}{2},$$
which is an algebraic number of degree two over~$\mathbb{Q}$. Indeed, we have 
$$\alpha^2 > \alpha^2 - 2\alpha + 5 > (\alpha - 1) ^2$$ 
in case~$\alpha \ne 1,2$, so this number is not a square and~$\mu^2$ is not rational. Neither is it in case~$\alpha = 2$, by direct calculation, 
and the case~$\alpha = 1$ is not needed. 

We are now ready to apply Theorem~\ref{thm:criterion}. Let~$\Omega'$ be the matrix obtained from~$\Omega$ by deleting 
the row and the column corresponding to the cylinder with~$2n + 1$ or~$2n+2$ squares. 
We have~$$\sigma(\Omega + 2I) \ge \sigma(\Omega' + 2I) - 1.$$ By construction,~$\Omega'$ is the adjacency matrix of a forest 
consisting of path graphs, so that~$2I + \Omega'$ is positive definite. 
We get~$$\sigma(\Omega + 2I) \ge \dim(\Omega) - 2 >  \dim(\Omega) - 4.$$
The criterion applies and the mapping class~$T_\alpha \circ T_\beta$ is pseudo-Anosov with stretch factor~$\lambda$ of degree~$2d = 4$.

For the case~$d \ge 3$, we take the examples as in the proof of Theorem~\ref{thm:hyp}, without specialising~$y$. 
Let~$\Omega'$ be the matrix obtained from~$\Omega$ by deleting 
the rows and the columns corresponding to the horizontal cylinder with~$2n + 1$ or~$2n+2$ squares, 
and to the vertical cylinder with~$y + 1$ squares. 
 We have~$$\sigma(\Omega + 2I) \ge \sigma(\Omega' + 2I) - 2.$$ By construction,~$\Omega'$ is the adjacency matrix of a forest 
consisting of path graphs, so that~$2I + \Omega'$ is positive definite. 
We get~$$\sigma(\Omega + 2I) \ge \dim(\Omega) - 4 >  \dim(\Omega) - 2d.$$
The criterion applies and the mapping class~$T_\alpha \circ T_\beta$ is pseudo-Anosov with stretch factor~$\lambda$ of degree~$2d$.
\end{proof}

\subsection{Reaching hyperelliptic component of $\mathcal H(2g-2)$ with degree~$g$}

Take the stair case model with a ``long'' stair made of $y^2$ squares (see Figure~\ref{hyp:minimal}).
\begin{figure}[htbp]
\begin{tikzpicture}[scale=0.8]

\draw[thick] (3,0) -- (3,1) -- (4.5,1); \draw[thick] (5.5,1) -- (7,1);
\draw[thick] (3,0) -- (3,0) -- (4.5,0); \draw[thick] (5.5,0) -- (8,0)--(8,1);

\foreach \i in {0,1}
{\draw[thick] (7+\i ,1+\i) -- (7+\i,2+\i) -- (9+\i,2+\i) -- (9+\i,1+\i) -- cycle;

\draw[thick] (8+\i ,1+\i)--(8+\i ,2+\i);}
\draw[thick] (10,4) -- (10,5) -- (12,5) -- (12,4) -- cycle;
\draw[thick] (11,4) -- (11,6) -- (12,6) -- (12,5);
\draw[thick] (4,0) -- (4,1);  \draw[thick] (6,0) -- (6,1);
\draw[thick] (7,0) -- (7,1);

\foreach \p in {(3,1),(3,0),(7,1),(9,1),(9,3),(12,4),(11,4),(12,6),(11,5),(12,5),(9,2),(7,2),(7,0),(8,0),(8,1),(8,2),(8,3),(10,2),(10,3),(10,4),(10,5),(11,6)}
{
\draw[fill=black] \p circle (2.2pt);
}
\foreach \p in {(6,0),(6,1),(4,1),(4,0)}
{
\draw[fill=red] \p circle (2.2pt);
}

\draw (7.5,2) node[above]{$\scriptstyle B_1$}; \draw (7.5,0) node[below]{$\scriptstyle B_1$};
\draw (8.5,3) node[above]{$\scriptstyle B_2$}; \draw (8.5,1) node[below]{$\scriptstyle B_2$};
\draw (9.5,2) node[below]{$\scriptstyle B_{3}$}; \draw (10.5,5) node[above]{$\scriptstyle B_{g-2}$};
\draw (11.5,4) node[below]{$\scriptstyle B_{g-1}$}; \draw (11.5,6) node[above]{$\scriptstyle B_{g-1}$};

\draw (10,3.1) node[above]{$\iddots$}; 
\draw (5,0.2) node[above]{$\cdots$}; 

\end{tikzpicture}
\caption{A stair case template in the hyperelliptic component of $\mathcal H(2g-2)$.}
\label{hyp:minimal}
\end{figure}
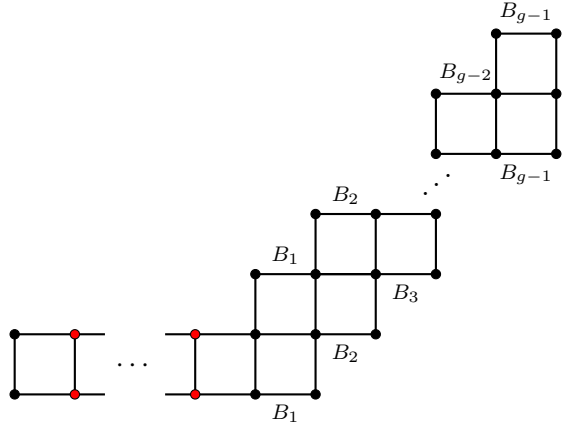
The $g\times g$ matrix is
\[
XX^\top=
\left(
\begin{array}{c|cccccc}
y^2 & 1 & \\
\hline 
1 & 2 & 1 &  \\
 & 1 & \ddots & \ddots &  \\
 &  & \ddots & \ddots & 1 \\
  &  &  & 1 & 2 & 1 \\
  & &  &  & 1 & 1   \\
\end{array}
\right).
\]
Let $p_g(t,y)$ be the characteristic polynomial of $XX^\top$. We will use the characteristic polynomial $q_g(t)$ of the $g\times g$ matrix $B_g$ obtained from $2I_{g} + \mathrm{Ad}(A_{g})$ by adding~$-1$ to the last diagonal entry, where $\mathrm{Adj}(A_g)$ is the adjacency matrix of the path graph with $g$ vertices.

By developing the determinant of $t\mathrm{Id}_{g}- XX^\top$ along the first column we get
$$
p_g(t,y)=-y^2q_{g-1}(t)+tq_{g-1}(t) - q_{g-2}(t).
$$ 
We now claim that the polynomials $q_{g-1}(t)$ and $tq_{g-1}(t) - q_{g-2}(t)$ are relatively prime. Using the same argument as in Lemma~\ref{charpoly:XXt:hyp}, we get that the matrix~$B_g$ is oscillatory, 
and hence the roots of~$p_{g-1}$ and~$p_{g-2}$ are all simple and pairwise distinct. 
Since the minimal polynomial of the Perron-Frobenius eigenvalue of~$B_g$ is a simple irreducible factor $F$ of $tq_{g-1}(t) - q_{g-2}(t)$ that is not also a factor of $q_{g-1}(t)$, Eisenstein's criterion applies and $p_g(t,y)$ is irreducible. Thus there are infinitely many specifications of $y>0$ such that $p_g(t,y) \in \mathbb Z[t]$ is irreducible. This yields the degree $d=g$ for the hyperelliptic component of $\mathcal H(2g-2)$ for any $g>1$.

Using this model, it is now straightforward to adapt the proofs of Theorem~\ref{thm:hyp} and Theorem~\ref{thm:hyp:stretch} to construct examples where the trace field is of degree~$g$ and the stretch factor is of degree~$2g$. 

\nocite{*}
\bibliography{biblio}

\providecommand{\bysame}{\leavevmode\hbox to3em{\hrulefill}\thinspace}
\providecommand{\MR}{\relax\ifhmode\unskip\space\fi MR }
\providecommand{\MRhref}[2]{%
  \href{http://www.ams.org/mathscinet-getitem?mr=#1}{#2}
}
\providecommand{\href}[2]{#2}
\begin{thebibliography}{{And}87}

\bibitem[{Ago}11]{Agol}
Ian {Agol}, \emph{{Ideal triangulations of pseudo-Anosov mapping tori}},
  Topology and geometry in dimension three. Triangulations, invariants, and
  geometric structures. Conference in honor of William Jaco's 70th birthday,
  Stillwater, OK, USA, June 4--6, 2010, Providence, RI: American Mathematical
  Society (AMS), 2011, pp.~1--17.

\bibitem[{And}87]{Ando}
Tsuyoshi {Ando}, \emph{{Totally positive matrices}}, {Linear Algebra Appl.}
  \textbf{90} (1987), 165--219.

\bibitem[AY81]{AY}
Pierre Arnoux and Jean-Christophe Yoccoz, \emph{Construction de
  diff{\'e}omorphismes pseudo-{Anosov}}, C. R. Acad. Sci., Paris, S{\'e}r. I
  \textbf{292} (1981), 75--78 (French).

\bibitem[CS13]{CS13}
Kariane Calta and Thomas~A. Schmidt, \emph{Infinitely many lattice surfaces
  with special pseudo-anosov maps}, J. Mod. Dyn. \textbf{7} (2013), no.~2,
  239--254.

\bibitem[GJ00]{Gutkin:Judge}
Eugene {Gutkin} and Chris {Judge}, \emph{{Affine mappings of translation
  surfaces: Geometry and arithmetic}}, {Duke Math. J.} \textbf{103} (2000),
  no.~2, 191--213.

\bibitem[GK37]{GK}
Felix {Gantmacher} and Mark-Grigorevich {Krein}, \emph{{Sur les matrices
  compl\`etement non n\'egatives et oscillatoires}}, {Compos. Math.} \textbf{4}
  (1937), 445--476.

\bibitem[KS00]{Kenyon:Smillie}
Richard {Kenyon} and John {Smillie}, \emph{{Billiards on rational-angled
  triangles}}, {Comment. Math. Helv.} \textbf{75} (2000), no.~1, 65--108.

\bibitem[KZ03]{KZ}
Maxim {Kontsevich} and Anton {Zorich}, \emph{{Connected components of the
  moduli spaces of Abelian differentials with prescribed singularities}},
  {Invent. Math.} \textbf{153} (2003), no.~3, 631--678.

\bibitem[{Lan}60]{Lang}
Serge {Lang}, \emph{{Le th\'eor\`eme d'irr\'eductibilit\'e de Hilbert}},
  {S\'emin. Bourbaki 12 (1959/60), Exp. No. 201, 13 pp. (1960).}, 1960.

\bibitem[{Lie}17]{Ldiss}
Livio {Liechti}, \emph{On the spectra of mapping classes and the 4-genera of
  positive knots}, 2017, Dissertation, Universit\"at Bern.

\bibitem[LP21]{Liechti:Pankau}
Livio {Liechti} and Joshua {Pankau}, \emph{{The geometry of bi-Perron numbers
  with real or unimodular Galois conjugates}}, {Int. Math. Res. Not.} (2021),
  rnab235.

\bibitem[{Pan}20]{Pankau}
Joshua {Pankau}, \emph{{Salem number stretch factors and totally real fields
  arising from Thurston's construction}}, {Geom. Topol.} \textbf{24} (2020),
  no.~4, 1695--1716.

\bibitem[{Pen}88]{Penner}
Robert~C. {Penner}, \emph{{A construction of pseudo-Anosov homeomorphisms}},
  {Trans. Am. Math. Soc.} \textbf{310} (1988), no.~1, 179--197.

\bibitem[PP90]{PP}
Athanase {Papadopoulos} and Robert~C. {Penner}, \emph{{Enumerating
  pseudo-Anosov foliations}}, {Pac. J. Math.} \textbf{142} (1990), no.~1,
  159--173.

\bibitem[{Rau}79]{Rauzy}
G\'erard {Rauzy}, \emph{{\'Echanges d'intervalles et transformations
  induites}}, {Acta Arith.} \textbf{34} (1979), 315--328.

\bibitem[{Str}17]{Strenner:algebraic}
Bal\'azs {Strenner}, \emph{{Algebraic degrees of pseudo-Anosov stretch
  factors}}, {Geom. Funct. Anal.} \textbf{27} (2017), no.~6, 1497--1539.

\bibitem[{Str}18]{Strenner:saf}
\bysame, \emph{{Lifts of pseudo-Anosov homeomorphisms of nonorientable surfaces
  have vanishing SAF invariant}}, {Math. Res. Lett.} \textbf{25} (2018), no.~2,
  677--685.

\bibitem[{Thu}88]{Th}
William~P. {Thurston}, \emph{{On the geometry and dynamics of diffeomorphisms
  of surfaces}}, {Bull. Am. Math. Soc., New Ser.} \textbf{19} (1988), no.~2,
  417--431.

\bibitem[{Thu}14]{Thurston:Last}
\bysame, \emph{{Entropy in dimension one}}, Frontiers in complex dynamics. In
  celebration of John Milnor's 80th birthday. Based on a conference, Banff,
  Canada, February 2011, Princeton, NJ: Princeton University Press, 2014,
  pp.~339--384.

\bibitem[{Vee}82]{Veech}
William~A. {Veech}, \emph{{Gauss measures for transformations on the space of
  interval exchange maps}}, {Ann. Math. (2)} \textbf{115} (1982), 201--242.

\bibitem[{Vee}89]{Veech:construction}
\bysame, \emph{{Teichm\"uller curves in moduli space, Eisenstein series and an
  application to triangular billiards}}, {Invent. Math.} \textbf{97} (1989),
  no.~3, 553--583.

\bibitem[{Yaz}21]{Yazdi}
Mehdi {Yazdi}, \emph{{Lower bound for the Perron-Frobenius degrees of Perron
  numbers}}, {Ergodic Theory Dyn. Syst.} \textbf{41} (2021), no.~4, 1264--1280.

\end{thebibliography}
\bibliographystyle{amsalpha}

\end{document}